\documentclass[a4paper,10pt]{article}

\RequirePackage[OT1]{fontenc}
\RequirePackage{amsthm,amsmath}
\RequirePackage[numbers]{natbib}
\RequirePackage[colorlinks,citecolor=blue,urlcolor=blue]{hyperref}

\usepackage{url}
\usepackage{graphicx}
\usepackage{amssymb}
\usepackage{amscd}
\usepackage{comment}
\usepackage{stmaryrd}
\usepackage{xcolor}
\usepackage{enumerate}
\usepackage{array}
\usepackage{thmtools}
\usepackage[linesnumbered,ruled,vlined]{algorithm2e}

\usepackage{lscape}
\usepackage{authblk}
\usepackage{subcaption}
\usepackage{caption}

\newtheorem{theorem}{Theorem}
\newtheorem{lemma}[theorem]{Lemma}
\newtheorem{hyp}{Assumption}
\newtheorem{proposition}{Proposition}
\newtheorem{definition}{Definition}
\theoremstyle{remark}

\newtheorem{example}{Example}

\newcommand{\X}{{\cal X}}
\newcommand{\Y}{{\cal Y}}

\def\S{{\cal S}}

\newcommand{\V}{{\cal V}}
\newcommand{\Z}{{\cal Z}}

\newcommand{\cH}{{\cal H}}

\newcommand{\argmax}{\mathop{\rm arg\, max}}

\newcommand{\RR}{\mathbb{R}}
\newcommand{\NN}{\mathbb{N}}
\newcommand{\EE}{\mathbb{E}}
\newcommand{\roc}{\rm ROC}
\newcommand{\auc}{\rm AUC}

\newcommand{\bX}{\mathbf{X}}
\newcommand{\bx}{\mathbf{x}}
\newcommand{\by}{\mathbf{y}}
\newcommand{\bY}{\mathbf{Y}}
\newcommand{\bZ}{\mathbf{Z}}
\newcommand{\bF}{F}

\newcommand{\dd}{\mathrm{d} }
\newcommand{\iid}{\textit{i.i.d.}}
\newcommand{\cdf}{\textit{c.d.f.}\;}

\newcommand{\rv}{\textit{r.v.}}
\newcommand{\ie}{\textit{i.e.}}

\newcommand{\wrt}{\textit{w.r.t.}}
\newcommand{\st}{\textit{s.t.}\;}
\newcommand{\eg}{\textit{e.g.}}

\newcommand{\resp}{\textit{resp.}}
\newcommand{\IFF}{\textit{iff.}\;}
\newcommand{\vc}{{\sc VC}}

\title{On Ranking-based Tests of Independence}
\date{}

\author[1]{Myrto Limnios\thanks{Corresponding author.} }
\author[2]{Stephan Cl\'emen\c{c}on}

\affil[1]{\small  \url{myli@math.ku.dk}\\
Department of Mathematical Sciences\\
University of Copenhagen, Universitetsparken 5, Copenhagen, 
 2100, Denmark.}

\affil[2]{ \small \url{stephan.clemencon@telecom-paris.fr}\\
	Telecom Paris, LTCI, Institut Polytechnique de Paris\\
	19 place Marguerite Perey, Palaiseau, 91120, France.}

\begin{document}

\maketitle

\begin{abstract}
    In this paper we develop a novel nonparametric framework to test the independence of two random variables $\bX$ and $\bY$ with unknown respective marginals $H(\dd x)$ and $G(\dd y)$ and joint distribution $F(\dd x\dd y)$, based on {\it Receiver Operating Characteristic} ($\roc$) analysis and bipartite ranking. The rationale behind our approach relies on the fact that, the independence hypothesis $\mathcal{H}_0$ is necessarily false as soon as the optimal scoring function related to the pair of distributions $(H\otimes G,\; F)$, obtained from a bipartite ranking algorithm, has a $\roc$ curve that deviates from the main diagonal of the unit square.
    We consider a wide class of rank statistics encompassing many ways of deviating from the diagonal in the $\roc$ space to build tests of independence. Beyond its great flexibility, this new method has theoretical properties that far surpass those of its competitors. Nonasymptotic bounds for the two types of testing errors are established. From an empirical perspective, the novel procedure we promote in this paper exhibits a remarkable ability to detect small departures, of various types, from the null assumption $\cH_0$, even in high dimension, as supported by the numerical experiments presented here. 
\end{abstract}

\section{INTRODUCTION}\label{sec:introduction}
Let $(\mathbf{X}_1,\; \mathbf{Y}_1),\; \ldots,\; (\mathbf{X}_N,\; \mathbf{Y}_N)$ be $N\geq 1$ independent and identically distributed ($\iid$) random pairs, defined on a space $(\Omega,\; \mathcal{F},\; \mathbb{P})$ and valued in the product space $\mathcal{X}\times \mathcal{Y}$, copies of the generic random pair $(\mathbf{X},\; \mathbf{Y})$. An important problem, occurring in many applications, consists in testing the independence of the two $\rv$'s $\mathbf{X}$ and $\mathbf{Y}$ based on the observation of the $(\mathbf{X}_i,\; \mathbf{Y}_i)$'s. It is considered here from a nonparametric perspective, meaning that no assumptions are made about the distribution $F(\dd x\dd y)$ of the pair $(\mathbf{X},\; \mathbf{Y})$, nor about the marginal distributions $H(\dd x)$ and $G(\dd y)$ of $\mathbf{X}$ and $ \mathbf{Y}$. The goal is to test the \textit{composite hypothesis}: 
\begin{equation} \label{eq:problem}
\mathcal{H}_0:\;\; F=H\otimes G\;\;  \text{ versus }\;\;  \mathcal{H}_1:\;\; F\neq H\otimes G~.
\end{equation}
The problem thus consists in testing whether two probability distributions on the product space $\mathcal{X}\times \mathcal{Y}$ are equal or not. Under additional (parametric) assumptions on the distribution $F$ (\textit{e.g.} discreteness, Gaussianity), various measures of dependence can be classically used to build pivotal test statistics (\textit{e.g.} chi-square statistic, empirical linear correlation). In the nonparametric case, most techniques consists in computing a statistical version of a (pseudo-) distance between $F$ and $H\otimes G$ (\textit{e.g.} integral probability metrics, see \cite{Rachevbook}). Refer to $\eg$ \cite{Szekely07,SzRi13} for covariance-based distances, generalized to metric spaces in \cite{Lyons13,jakobsen2017distance}.  \cite{GrettonBousquet05,GrettonSmola05,GBRSS07} introduced  kernel-based extensions relying on the \textit{Hilbert-Schmidt Independence Criterion} (HSIC), where the covariance distance being shown to be a specific instance of the class of HSIC-type measures of dependence in \cite{SejGrett13}. Other measures for testing independence have been recently proposed, see in particular, \cite{BerrSam19,Gonzalez2021DataDrivenRF} using the notion of mutual information, \cite{GrettonGyo2010,Heller16jmlr} based on partitioning techniques, and \cite{Reshef11,Reshefjmlr16,Reshef18} considering use of the maximal information criterion.

\paragraph{Rank statistics for testing independence.} The approach developed here, of completely different nature, is inspired by \textit{rank-based} methods (\cite{HajSid67} or \cite{Kallenberg}, \cite{KallenbergLedwina} or \cite{KallenbergLedvina2}) tailored to the situations where $\mathcal{X}= \mathcal{Y}=\mathbb{R}$ and $\mathcal{H}_0$ is tested against specific alternatives of \textit{positive (regression) dependence}\footnote{Two real-valued $\rv$'s $X$ and $Y$ defined on the same space exhibit \textit{positive dependence} \textit{iff} $\mathbb{P}\left(X>x,\; Y>y\right)\geq \mathbb{P}\left(X>x\right)\times \mathbb{P}\left(Y>y\right)$ for any $(x,y)\in \mathbb{R}^2$.}. Assuming in addition that $\bX$ and $\bY$ are continuous $\rv$'s, a natural strategy (see \cite{Kendall75}) consists in ranking the pairs $(\bX_i,\; \bY_i)$ according to increasing values of the $\bX_i$'s: $(\bX_{\sigma(1)},\; \bY_{\sigma(1)}),\; \ldots,\; (\bX_{\sigma(N)},\; \bY_{\sigma(N)})$, where $\sigma$ is the permutation of the index set $\{1,\; \ldots,\; N\}$ (\textit{i.e.} the element of the symmetric group $\mathfrak{S}_N$) $\st$ $\bf{X}_{\sigma(1)}<\ldots< \bf{X}_{\sigma(N)}$ and analyzing the ranks of the $\bY_{\sigma(i)}$'s through the rank correlation coefficient, see \textit{e.g.} Chapter 6 in \cite{LehmannRomano}: conditioned upon $(\bX_{\sigma(1)},\; \ldots,\; \bX_{\sigma(N)})$, the latter being uniformly random under $\mathcal{H}_0$, while the rank of $\bY_{\sigma(i)}$ among the $\bY_{\sigma(j)}$'s exhibits an `upward trend' under the positive dependence alternative (\textit{i.e.} it is stochastically increasing with $i$).
The approach to independence testing based on statistical learning we propose 
shares similarities with such rank-based techniques, it also consists ranking pairs in $\mathcal{X}\times \mathcal{Y}$. Extension of rank-based techniques for independence testing to multivariate data has been recently the subject of much attention in the literature. The sole approach enjoying distribution-freeness under nonparametric assumptions so far, is based on the notion of center-outward ranks/signs in \cite{Hallin17centerOut}. It is used in \cite{HalDrton22} to build \textit{generalized symmetric (test) statistics}: it boils down to plugging into classic statistics, e.g. the distance covariance measure for independence testing, a center-outward generalization of rank statistics by mapping any absolute continuous distribution to the spherical uniform distribution on the $d$-dimensional unit ball, solution of the related optimal transport formulation, see \cite{Hall21}. This encompasses the main modern rank-based and distance-based methods for testing the hypothesis of independence, see  $\eg$ \cite{DebSen19,LDrton18}. %, and generalized in \cite{HalDrton22}. 
While these methods have appealing theoretical properties, see \cite{HalDrton22}, they are limited by the strong negative impact of $d$ of the feature on their power, studied for kernel and distance based techniques,
see section 3 in \cite{RamdasAAAI15}. As shown in \cite{HXDGustat23}, this is caused by the dependence of the kernel of the $U$-statistic of degree two $\wrt$ the dimension $d$.

\paragraph{Our contributions.} The nature of our approach is quite different. It involves a preliminary statistical learning step, namely bipartite ranking on $\mathcal{X}\times \mathcal{Y}$, and relies on {\it Receiver Operating Characteristic} ($\roc$) analysis. The $\roc$ curve is the gold standard to differentiate between two univariate distributions. The rationale behind our methodology lies in the fact that, under the null hypothesis $\mathcal{H}_0$,  the {\it optimal} $\roc$ curve related to the bipartite ranking defined by the pair $(H\otimes G,\; F)$ of distributions on $\mathcal{X}\times \mathcal{Y}$, is known and coincides with the main diagonal of $[0,1]^2$. It is thus natural to quantify the departure from $\mathcal{H}_0$, by the deviations of the optimal $\roc$ curve from the diagonal.
%As will be shown, t
The latter can be summarized by appropriate \textit{two-sample (linear) rank statistics}, whose concentration properties have been %recently 
investigated in \cite{CleLimVay21}. Since the optimal $\roc$ curve is unknown in practice, a bipartite ranking task on the product space $\mathcal{X}\times \mathcal{Y}$ must be completed first %in order 
to rank the pairs.
Our method is implemented in three steps: after splitting the sample in two parts ($2$-split trick) and shuffling the pairs, a first part of the sample is used to train a bipartite ranking function to output a scoring function. The second part is then ranked using the scoring function previously learned so as to compute a test statistic assessing the possible departure from independence.
It may be applied in a general multivariate framework and has considerable advantages in the high-dimension case, compared to all its competitors, especially those based on probability metrics between statistical versions of $F$ and $H\otimes G$, see \textit{e.g.} \cite{GrettonGyo2010}. 
In contrast, provided that the model bias (\textit{i.e.}, the error inherent in the choice of the set of ranking functions over which the learning step is performed) is `small', the power of the test proposed is possibly affected by the dimension only through the choice of %the  model bias inherent in the 
bipartite ranking algorithm. This is supported here by a sound (nonasymptotic) theoretical analysis based on the concentration results for two-sample $R$-processes proved in \cite{CleLimVay21} and promising empirical results. Our method is shown to work well, in the vicinity of independence especially, surpassing the existing methods. 

\paragraph{Connection to the two-sample problem.} We point out that the use of (an estimate of) the optimal $\roc$ curve, on which the novel independence testing method promoted here relies, has been recently exploited for the purpose of statistical hypothesis testing in \cite{CleLimVay21test} to solve the two-sample problem, \textit{i.e.} to test the assumption that two $\iid$ samples share the same distribution. The major difference naturally lies in the nature of the alternatives to the null assumption, \textit{i.e.} departure from independence \textit{vs.} departure from homogeneity, but also in the statistical framework/analysis: whereas independent observations drawn from each of the two distributions to be tested equal are supposedly available in the two-sample problem, no sample of the distribution $H\otimes G$ is directly available under $\mathcal{H}_1$ in independence testing. A \textit{shuffling} procedure (\textit{i.e.} a random permutation of parts of the indices $\{1,\; \ldots,\; N\}$), that aims at building independent observations drawn from $H\otimes G$, is key in the testing method we propose  and analyze here. To summarize, for the method proposed here, new to the literature, we show that:
1) the test statistic is  distribution-free resulting in the exact computation of the testing threshold, 
2) a nearly optimal control of the type-II error with explicit parameters can be obtained for all types of alternative, 
%\item T
3) the method depends on the dimension of the underlying spaces only through the bipartite ranking algorithm, importantly avoiding any mispessification of the asymptotic distribution (\cite{HXDGustat23}) and harmfull high-dimensional setting.
The article is organized as follows. Section \ref{sec:background} recalls key notions pertaining to $\roc$ analysis and bipartite ranking, providing an insight into the rationale of the method. It is described and theoretically analyzed from a nonparametric and nonasymptotic perspective in section \ref{sec:main}. Numerical results are displayed in section \ref{sec:exp}, while concluding remarks are collected in section \ref{sec:conclusion}. Due to space constraints, some properties related to $\roc$ analysis and bipartite ranking, all technical details and proofs, as well as additional numerical experiments, are postponed to the Supplementary Material.
\section{PRELIMINARIES}\label{sec:background}
We first briefly recall the main concepts related to {\rm ROC} analysis and bipartite ranking, involved in the methodology subsequently proposed and analyzed. The rationale behind the latter is next explained. Here and throughout, by $\mathbb{I}\{ \mathcal{E} \}$ is meant the indicator function of any event $\mathcal{E}$,
 by $\delta_a$ the Dirac mass at any point $a$, by $W^{-1}(u)=\inf\{t\in (-\infty,\;+\infty]:\; W(t)\geq u\}$, $u\in [0,1]$ the generalized inverse of any cumulative distribution function $W(t)$ on $\RR\cup\{+\infty\}$. 
The floor and ceiling functions are denoted by $u \in \mathbb{R}\mapsto \lfloor u \rfloor$ and by $u \in \mathbb{R}\mapsto \lceil u \rceil$ respectively. For any bounded function $\psi:(0,1)\rightarrow \mathbb{R}$, we also set $\vert\vert \psi\vert\vert_{\infty}=\sup_{u\in (0,1)}\vert \psi(u)\vert$. We consider $\rv$ denoted in bold symbols as valued in a multivariate space $\mathcal{Z}$, $\eg$ subset of $\RR^d$, with $d\geq 2$.

\subsection{Bipartite Ranking and {\rm ROC} Analysis}\label{ssec:BPROC}

We explain the connection between bipartite ranking and the quantification of the discrepancy between two probability distributions on a same space.
\paragraph{$\roc$ analysis.} The $\roc$ curve is a gold standard to measure the difference between two \textit{univariate} distributions, $F_1$ and $F_2$ say. It is defined by the Probability-Probability plot  
$t\in\mathbb{R}\mapsto (1 - F_1(t),\; 1 - F_2(t))$,  connecting possible jumps by line segments by convention. It can alternatively 
be seen as the graph of a c\`ad-l\`ag (\textit{i.e.} right-continuous and left-limited) non-decreasing mapping defined by $u \in (0,1) \mapsto  \roc_{F_1,F_2}(u):=1-F_2\circ F_1^{-1}(1-u)$ 
at points $\alpha$ such that $F_2\circ F_1^{-1}(1-u)=1-u$. 
The curve $\roc_{F_1,F_2}$ %$\roc((F_1,F_2),\; .)$
coincides with the main diagonal of $[0,1]^2$ \textit{iff} $F_1=F_2$. Hence, the notion of $\roc$ curve offers a visual tool to examine the differences between two univariate distributions. For instance, the univariate distribution $F_2$ is stochastically larger\footnote{Recall that $F_2$ is said to be stochastically larger than $F_1$ \textit{iff} $F_1(t) \geq F_2(t)$  %$F_1(]t,\; +\infty[) \leq F_2(]t,\; +\infty[)$ 
for all $t\in \mathbb{R}$.} than $F_1$ \textit{iff} the curve $\roc_{F_1,F_2} \; $  is everywhere above the main diagonal. Of course, the curve $\roc_{F_1,F_2}$%$\roc((F_1,F_2),\; .)$
 is unknown in practice, just like the $F_i$'s. 
Hence, $\roc$ analysis must be based on independent \textit{i.i.d.} samples $(X_{1,1},\; \ldots,\; X_{1,n_1})$ and $(X_{2,1},\; \ldots,\; X_{2,n_2})$ with distributions $F_1$ and $F_2$ respectively and consists in plotting $\roc_{\hat{F}_1,\hat{F}_2}$, %$\roc((\hat{F}_1,\hat{F}_2),\; .)$
 where $\hat{F}_i=(1/n_i)\sum_{k\leq n_i}\delta_{X_{i,k}}$ is the corresponding empirical counterpart of $F_i$ with $i\in\{1,\; 2\}$.
A popular scalar summary %of the $\roc$ functional % curve $\roc_{F_1,F_2}$ %$\roc((F_1,F_2),\; .)$
 %is the Area Under the $\roc$ Curve criterion ($\auc$ in abbreviated form) 
 is the Area Under the $\roc$ Curve ($\auc$), defined by $\auc(F_1,F_2)=\int_0^1\roc_{F_1, F_2}(u)du$. %$\auc(F_1,F_2)=\int_0^1\roc((F_1,F_2),\; \alpha)d\alpha$. 
 Its empirical version can be expressed as an affine transform of a (two-sample linear) rank statistic, the Mann-Whitney Wilcoxon (MWW) statistic $\hat{W}_{n_1,n_2}=\sum_{k\leq n_2}R(X_{2,k})$, where the ranks $R(X_{2,k})=\sum_{l\leq n_1} \mathbb{I}\{X_{1,l} \leq X_{2,k}\} + \sum_{l\leq n_2} \mathbb{I}\{X_{2,l}\leq X_{2,k}\}$ denotes the rank of $X_{2,k}$ among the pooled sample: 
\begin{equation}\label{eq:auc}
n_1n_2\auc(\hat{F}_1,\hat{F}_2)=\hat{W}_{n_1,n_2}- \frac{n_2(n_2+1)}{2}~.
\end{equation}
It is thus a distribution-free statistic (concentrated around the value $1/2$) when $F_1=F_2$, that can be naturally used to test the hypothesis of equality in distribution based on the $X_{i,k}$'s with $i\in \{1,\; 2\}$.

\paragraph{Bipartite ranking.} Consider now two distributions $F_+$ and $F_-$ on a general measurable space $\mathcal{Z}$, referred to as positive and negative distributions. 
Let two independent $\iid$ samples $\bX_{+,1},\; \ldots,\; \bX_{+,n_+}$ and $\bX_{-,1},\; \ldots,\; \bX_{-,n_-}$ drawn from $F_+$ and $F_-$ respectively. The goal of bipartite ranking is to learn a scoring function $s:\mathcal{Z}\rightarrow (-\infty,\; \infty]$, based on the two samples, to rank any new observation without prior knowledge, by inducing a total preorder on $\mathcal{Z}$ statistically ranking the positive instances ($+$) at the top of the resulting list compared to the negative ones ($-$), $\ie$, $\forall (x,x') \in \mathcal{Z}^2$, $x\preccurlyeq_s x'$ \textit{iff} $s(x)\leq s(x')$. Let $\S$ be the set of all scoring functions on $\mathcal{Z}$. One evaluates the ranking performance of a candidate $s(z)$ in $\S$ by plotting (a statistical version of) the $\roc$ curve $\roc((F_{s,-},F_{s,+}),\; \alpha)=\roc(s,\alpha)$, denoting by $F_{s,\epsilon}$ %and $G_s$ 
the pushforward distribution of $F_{\epsilon}$ by the mapping $s(z)$ for $\epsilon\in\{-,\; +\}$. This  defines a partial preorder on $\mathcal{S}$: for all %scoring functions 
$(s_1, \; s_2)$, $s_2$ is more accurate than $s_1$ when $\roc(s_1,\cdot)\leq \roc(s_2,\cdot)$ on $[0,1]$. % for all $\alpha\in [0,1]$. 
The most accurate scoring functions are increasing transforms of the likelihood ratio $\Psi(z)=\dd F_{s,+}/\dd F_{s,-}(z)$, %$\Psi(z)=\dd G/\dd H(z)$
 as can be deduced from a straightforward Neyman-Pearson argument (see \textit{e.g.} Proposition 4 in \cite{CV09ieee}): $\S^*=\left\{s\in \S, \forall (z,\; z' ) \in\mathcal{Z}^2, \Psi(z)<\Psi(z') \Rightarrow s^*(z)<s^*(z') \right\}$. 
For all $(s,\; u)\in \S\times (0,1)$, we have: $\roc(s,\; u)\leq \roc^*(u)$, where $\roc^*(\cdot)=\roc(\Psi,\; \cdot)=\roc(s^*,\; \cdot)$ for any $s^* \in \S^*$. The optimal curve is always concave, increasing, above the main diagonal of the $\roc$ space consequently, \textit{cf} \cite{CV09ieee}. A key to understanding the method in section \ref{sec:main} is to realize that $F_+=F_-$ \textit{iff} $\roc^*$ coincides with the diagonal of $[0,1]^2$, see subsection \ref{subsec:rationale}.
 
\paragraph{$\roc$ curve optimization.} From a quantitative perspective, bipartite ranking aims at building a scoring function $s(z)$, based on the $\bX_{\epsilon,k}$'s with a $\roc$ curve as close as possible to $\roc^*$. A typical way of measuring the deviation between these curves is to consider their distance in $\sup$ norm.
As $\roc^*$ is unknown, just like $\S^*$, no straightforward statistical counterpart of this loss can be computed. In \cite{CV09ieee} and \cite{CV10CA}, it is %has been %shown 
 proved that bipartite ranking can be viewed as nested cost-sensitive classification tasks. By % and %somehow
   discretizing them adaptively, empirical risk minimization can be sequentially applied, with statistical guarantees in the $\sup$-norm sense at the cost of an approximation bias.
 Ranking performance can  be also measured by means of the $L_1$-norm in the $\roc$ space:
$\int_{0}^1\vert \roc(s,u)-\roc^*(u)  \vert \dd u =\auc^*-\auc(s)$, where $\auc(s)=\auc(F_{s,-},F_{s,+})$ and $\auc^*=\auc(\Psi)$.
 The minimization of the $L_1$-distance to $\roc^*$ is equivalent to the maximization of the (scalar) $\auc$ criterion. Maximizing the latter %Maximizing statistical counterparts of the $\auc$ criterion 
 over a class $\S_0\subset\S$, of controlled complexity, is a popular approach to bipartite ranking, and documented in various articles. Refer to \textit{e.g.} \cite{AGHHPR05} or \cite{CLV08} for upper confidence bounds for the $\auc$ deficit of scoring rules obtained by solving
%\begin{equation*}
$\max_{s\in \S_0} \; \auc(\hat{F}_{s,-},\hat{F}_{s,+})$,
%\end{equation*}
where  $\widehat{F}_{s,\epsilon}=(1/n_{\epsilon})\sum_{j=1}^{n_{\epsilon}}\delta_{s(\bX_{\epsilon,j})}$ for $\epsilon\in\{-,+\}$. As noticed in \eqref{eq:auc}, this boils down to maximizing the rank-sum criterion:
%\begin{equation*}
$\hat{W}_{n_-,n_+}(s)=\sum_{i=1}^{n_+}R(s(\bX_{+,i}))$,
%\end{equation*}
where $R(s(\bX_{+,i})) =N\hat{F}_{s,N}(s(\bX_{+,i}))$  for $i\in\{1,\; \ldots,\; n_+  \}$, $\hat{F}_{s,N}(t)=(1/N)\sum_{\epsilon\in\{-,+\}}\sum_{i=1}^{n_{\epsilon}}\mathbb{I}\{s(\bX_{\epsilon, i}) \leq t  \}$ for $t\in \mathbb{R}$ and $N=n_++n_-$. As expected, appropriate ranking performance criteria take  the form of \textit{(two-sample linear) rank  statistics}, see \cite{CV07}. In \cite{CleLimVay21}, the empirical ranking performance measures
\begin{equation}\label{eq:crit_emp}
    \hat{W}^{\phi}_{n_-,n_+}(s)=\sum_{i=1}^{n_+}\phi\left( \frac{R(s(\bX_{+,i}))}{N+1} \right)~,
\end{equation}
where $\phi:[0,1]\to \mathbb{R}$ is an increasing \textit{score-generating function} that %rules the weights assigned to the 
weights the positive ranks involved the functional, are considered. For $\phi(v)=v$, one recovers the MWW statistic and the $\auc$ criterion, see \eqref{eq:auc}. If $F_{s,+}=F_{s,-}$, the ranks of the `positive scores' are uniformly distributed. The distribution $\mathcal{L}^{\phi}_{n_-,n_+}$ of \eqref{eq:crit_emp} is thus independent from the distributions of the $\bX_{\epsilon, i}$'s, and can be tabulated by means of elementary combinatorial computations.
When $n_+=\lfloor pN\rfloor $ and $n_-=\lceil (1-p)N\rceil $ for $p\in (0,1)$, the statistic $(1/N)\hat{W}^{\phi}_{n_-,n_+}(s)$ can be viewed as
an empirical version of $W_{\phi}$-ranking performance:
\begin{multline}\label{eq:W_crit}
W_{\phi}(s)=\mathbb{E}\left[ (\phi \circ F_s)(s(\bX_+)) \right] = \frac{1}{p}\int_{0}^1\phi(v)\dd v \\
-\frac{1-p}{p}\int_{0}^1 \phi\left(p(1-\roc(s,\; \alpha))+(1-p)(1-u)\right)\dd u~,
 \end{multline}
 where $F_s=pF_{s, +}+(1-p)F_{s, -}$ for any $s\in \S$. For any score-generating function $\phi$ that rapidly vanishes near $0$ and takes much higher values near $1$, such as $\phi(v)=v^q$ with $q>1$, the quantity \eqref{eq:W_crit} reflects the behavior of the curve  $\roc(s,\; \cdot)$ near $0$, $\ie$, the probability that $s(\mathbf{X}_+)$ takes the highest values in other words.  As stated in Proposition 6 of  \cite{CleLimVay21}, for any $s, \; s^*\in \S\times  \S^*$, we have $W_{\phi}(s)\leq W^*_{\phi}:= W_{\phi}(\dd F_+ / \dd F_-)=W_{\phi}(s^*)$. If $\phi$ is strictly increasing, %on $(0,1)$, the ensemble 
$\S^*$ coincides with the ensemble of maximizers of $W_{\phi}$. 
In \cite{CleLimVay21}, bounds for the maximal deviations between \eqref{eq:crit_emp} and $NW_{\phi}(s)$ over appropriate classes $\mathcal{S}_0$ have been proved, and generalization results for maximizers of the empirical $W_{\phi}$-ranking performance criterion based on the latter have been established. 
The theoretical analysis carried out subsequently relies on these results.

\subsection{On Dependence through {\rm ROC} Analysis}\label{subsec:rationale}
We now go back to the problem recalled in  section \ref{sec:introduction} and explain why the analysis of $\roc$ curves and their scalar summaries \eqref{eq:W_crit} provide natural tools to test the statistical hypothesis of independence $\mathcal{H}_0$. Consider the notations introduced in section \ref{ssec:BPROC}, and set $\mathcal{Z}=\mathcal{X}\times\mathcal{Y}$, $F_-=H\otimes G$ and $F_+=F$. Our approach relies on the observation that deviations of the curve $\roc^*$ from the main diagonal of $[0,1]^2$, as well as those of $W^*_{\phi}$ from $\int_0^1\phi(v)\dd v$, for appropriate score-generating functions $\phi$,  provide a natural way of measuring the departure from $\mathcal{H}_0$, as revealed by the theorem below.

\begin{theorem}\label{thm:rationale}
The following assertions are equivalent.
\begin{itemize}
    \item[]$(i)$ The hypothesis `$\mathcal{H}_0:\; H\otimes G=F$' holds true.
    \item[] $(ii)$  The optimal $\roc$ curve relative to the bipartite ranking problem defined by the pair $(H\otimes G, F)$ coincides with the diagonal of $[0,1]^2$: 
    $\forall u\in (0,1),\;\; \roc^*(u)=u$.
    \item[] $(iii)$  For any function $\phi(v)$, we have
    $W^*_{\phi}=\int_0^1\phi(v)\dd v~.$
    \item[] $(iv)$  There exists a strictly increasing score-generating function $\phi(u)$ s.t.
    $W^*_{\phi}=\int_0^1\phi(v)\dd v$.
    \item[] $(v)$  We have $\auc^*=1/2$.
\end{itemize}
In addition, we have:
    \begin{equation}\label{eq:aucdist}
    \auc^*-\frac{1}{2}=\int\int%_{\bx, \by\in\mathcal{X}\times \mathcal{Y}}
    \left\vert \frac{dF}{d(H\otimes G)}(\bx, \by)-1 \right \vert H(d\bx)G(d\by)~.
    \end{equation}
\end{theorem}

 Hence, the optimal curve $\roc^*$ quantifies the dissimilarity between the  $H\otimes G$ and $F$, as depicted by Eq. \eqref{eq:aucdist}.

\begin{example}{\sc (Multivariate Gaussian Variables)}\label{exgauss} Consider a centered Gaussian r.v. $(\bX,\; \bY)$ with definite positive covariance $\Gamma$, valued in $\mathbb{R}^{q}\times \mathbb{R}^{l}$. Denote by $\Gamma_{\bX}$ and $\Gamma_{\bY}$ the (definite positive) covariance matrices of the components $\bX$ and $\bY$. As an increasing transform of the likelihood ratio, the quadratic scoring function
%\begin{equation*}
 $s : z\in \mathbb{R}^{q+l}\mapsto   z^t (\Gamma^{-1}-diag(\Gamma_{\bX}^{-1}, \Gamma_{\bY}^{-1}) )z$
%\end{equation*}
is optimal. 
When $\text{Cov}(X^1,Y^k) = \rho$, for all $k\leq l$, with $\rho\in[0,1)$,  and $\Gamma_{i,j} = \delta_{ij}$ otherwise, 
the hypothesis $\mathcal{H}_0$ is naturally true \textit{iff} $\rho = 0$. The optimal $\roc$ curve is plotted in Fig. \ref{fig:exampleGL} for different values of the parameter $\rho$, such that $\Gamma$ is positive definite, and $q=l=5$. We further refer to section \ref{appssec:gaussexample} for an advanced analysis in light of the proposed method.

\end{example}
\begin{figure}[ht!]
\centering
\begin{tabular}{cc}
\parbox{0.5\textwidth}{
\includegraphics[width=0.35\textwidth]{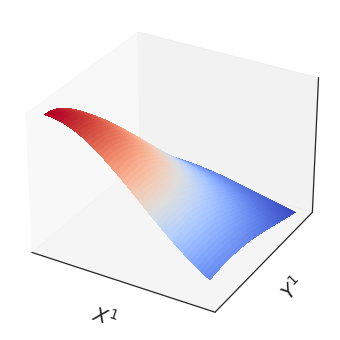}
} %\hspace{1cm}
\parbox{0.5\textwidth}{
\includegraphics[scale=0.3]{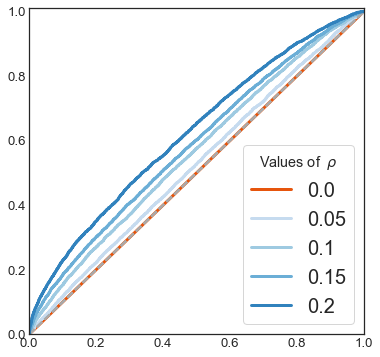}
}
\end{tabular}
\vspace{-0.2cm}
\caption{Left: Joint Gaussian density for $\rho=0.20$ of $(X^1, Y^1)$. Right: Plots of the optimal $\roc$ curves for two Gaussian vectors with linear correlation $\rho \in\{0.0, \; 0.05, \; 0.10, \; 0.15, \; 0.20\} $ and $q=l=5$.} 
    \label{fig:exampleGL} 
%\end{wrapfigure}
\end{figure}  
    
\paragraph{Ranking-based rank tests of independence.} Theorem \ref{thm:rationale} shows that the testing problem \eqref{eq:problem} can be reformulated in terms of properties of the optimal $\roc$ curve, related to the bipartite ranking problem $(H\otimes G, F)$, as
%\begin{equation}\label{eq:2_sample_pb_alt}
`$\mathcal{H}_0:\;\; \auc^*=1/2$ \textit{vs.} $\mathcal{H}_1: \auc^*>1/2$',
%\end{equation}
or, equivalently, as
%\begin{equation}\label{eq:2_sample_pb_alt2}
`$\mathcal{H}_0:\;\; W_{\phi}^*=\int_0^1\phi(u)\dd u$ \textit{vs.}  $\mathcal{H}_1: \;\; W_{\phi}^*>\int_0^1\phi(u)\dd u$', 
%\end{equation}
for any given strictly increasing score generating function $\phi(u)$. It is noteworthy that these formulations are \textit{unilateral}, the optimal $\roc$ curve being necessarily above the diagonal. From a practical perspective, the curve $\roc^*$ as well as its scalar summaries, such as $\auc^*$ or $W^*_{\phi}$, are unknown.
The approach we propose is thus implemented in three steps. After splitting the samples $\{(\bX_1,\bY_1)\;\ldots,\; (\bX_N,\bY_N)\}$ into two parts: 1) based on the first part, build two independent $\iid$ samples with respective distributions $H\otimes G$ and $F$, then 2) solve the corresponding bipartite ranking problem and produce a scoring function $\widehat{s}(z)$, as described above. Finally, 3) perform a univariate rank-based test  based on a statistic of type \eqref{eq:crit_emp} computed from the second part of the data, once scored using $\widehat{s}$, to detect possible statistically significant deviations between the $\roc$ curve and the diagonal.  

The subsequent sections provide both theoretical and empirical evidence that, beyond the fact that they are nearly unbiased, such testing procedures permit to detect very small departures, of various types, from the hypothesis of independence.

\section{METHODOLOGY AND THEORY}\label{sec:main}
We now describe at length the testing procedure previously sketched, and next establish the related theoretical guarantees by proving nonasymptotic bounds for the two types of testing error. Throughout this section, we set $F_-=H\otimes G$ and $F_+=F$. 
\subsection{Ranking-based Rank Test Statistics}\label{ssec:proc}
Following section \ref{sec:background}, 
two steps are required to implement the procedure proposed. Let $n<N$ be an even integer. Hence, we use a classic two-split trick to independently divide the original $\iid$ sample
$\{(\mathbf{X}_1,\bY_1)\; \ldots,\; (\mathbf{X}_N,\bY_N)\}$ into two: 
\begin{eqnarray*}
    \mathcal{D}_{n} &:=& \{ (\bX_i,\bY_i)  :\; i=1,\; \ldots,\; n\}\\
    \mathcal{D}'_{n'} &:=& \{ (\bX_i,\bY_i)  :\; i=n+1,\; \ldots,\; N\}~,
\end{eqnarray*}
with $n<N$ and $n'=N-n$. Fix $p\in (0,1)$, set $n_+=\lfloor pn\rfloor=n-n_-$ and $n'_+=\lfloor  pn'\rfloor=n'-n'_{-}$. Consider two independent random variables $\sigma$ and $\sigma'$, defined on the same probability space $(\Omega, \mathcal{F},\mathbb{P})$ as the $(\bX_i,\bY_i)$'s and independent of the latter, uniformly distributed in $\mathfrak{S}_{n_-}$ and $\mathfrak{S}_{n'_-}$ respectively. From the first part $\mathcal{D}_n$, one considers the two samples:  $  \mathcal{D}^-_{n_-}=\{(\bX_i, \bY_{\sigma(i)})_{ 1\leq i\leq  n_-}\},\quad \mathcal{D}^+_{n_+}=\{(\bX_i, \bY_{i})_{ 1+n_-\leq i\leq  n}\}$, whereas the samples below are formed from the second part  
        $\mathcal{D}'^{-}_{n'_-}=\{(\bX_i, \bY_{n+\sigma'(i-n)})_{ 1+n\leq i\leq n+ n'_-}\} , \quad \mathcal{D}'^{+}_{n'_+}=\{(\bX_i, \bY_{i})_{ 1+n+n'_-\leq i\leq  N}\}$.

\begin{proposition}\label{prop:build_samples} The following assertions hold true.

$(i)$ The samples $\mathcal{D}^-_{n_-}$, $\mathcal{D}^+_{n_+}$, $\mathcal{D}'^-_{n'_-}$, and $\mathcal{D}'^+_{n'_+}$ are independent.

$(ii)$ For $\epsilon\in\{-,+\}$, $\mathcal{D}^{\epsilon}_{n_{\epsilon}}$ and  $\mathcal{D}'^{\epsilon}_{n'_{\epsilon}}$ are $\iid$ samples  with distribution $F_{\epsilon}$.
\end{proposition}
Now that we are equipped with the two pairs of negative/positive samples constructed above,  the procedure we propose requires two ingredients: a bipartite ranking algorithm $\mathcal{A}$ that permits to construct a scoring function $\hat{s}=\mathcal{A}(\mathcal{D}^-_{n_-}, \mathcal{D}^+_{n_+})$ based on the first part of the data (see the algorithms in \textit{e.g.} \cite{FISS03}, \cite{Rak04}, \cite{RCMS05}, \cite{Rud06} or \cite{lambdarank07}) and a strictly increasing score-generating function $\phi$. As explained in section \ref{sec:background}, the independence testing problem \eqref{eq:problem} for the couple $(H\otimes G, F)$ can be expressed as follows:
\begin{equation}\label{eq:problem2}
    \mathcal{H}_0:\; W_{\phi}^*=\int_0^1\phi(u)\dd u \;\; \text{vs.} \;\;  \mathcal{H}_1: \; W_{\phi}^*>\int_0^1\phi(u)\dd u~.
\end{equation}
Notice that the formulation above is unilateral, the optimal curve $\roc^*$ being always above the first diagonal, or equivalently, the pushforward distribution of $F$ by $\Psi(x,y)$ is always stochastically larger than that of $H\otimes G$.
Relying on \eqref{eq:problem2}, one computes the values taken by the scoring function $\hat{s}(x,y)$ over the pooled data set $\mathcal{D}'^-_{n'_-}\cup \mathcal{D}'^+_{n'_+}$ and next the following version of the  statistic \eqref{eq:crit_emp}:
\begin{eqnarray}\label{eq:Whatind}
    \widehat{W}^{\phi}_{n'_-,n'_+}(\hat{s}, \sigma')=\sum_{i=1+n+n'_-}^N\phi\left( \frac{R'_{\sigma'}(\hat{s}(\bX_i,\bY_i))}{n'+1} \right)~,
\end{eqnarray}
where the ranks are defined on $\mathcal{D}'_{n'}$ by 
$R'_{\sigma'}(t) = \sum_{i=1+n+n'_-}^N \mathbb{I}\{\hat{s}(\bX_i,\bY_i) \leq t\} + \sum_{i=1+n}^{n+n'_-}$ $ \mathbb{I}\{\hat{s}(\bX_i,\bY_{n+\sigma'(i-n)}) \leq t\}$. %One straightforwardly deduces from Proposition \ref{prop:build_samples} that, 
Under $\mathcal{H}_0$, the test statistic \eqref{eq:Whatind} has distribution $\mathcal{L}^{\phi}_{n'_-,n'_+}$, similar to the univariate rank statistic defined in \eqref{eq:crit_emp} by Proposition \ref{prop:build_samples}. 
Fix now the desired level $\alpha\in (0,1)$ of the test of independence. Consider the $(1-\alpha)$-quantile $q^{\phi}_{n'_-,n'_+}(\alpha)$ of the pushforward distribution of $\mathcal{L}^{\phi}_{n'_-,n'_+}$, by the mapping $w\mapsto (1/n')w-\int_0^1\phi(u)\dd u$,  depending only on $\phi$, $n'_+$ and $n'_-$. %We can now define the test built by thresholding the statistic \eqref{eq:Whatind} with cut-off $q^{\phi}_{n'_-,n'_+}(\alpha)$, and of   
Proposition \ref{prop:testPhi} proves that constructing a test based on the statistic \eqref{eq:Whatind} and using this $(1-\alpha)$-quantile $q^{\phi}_{n'_-,n'_+}(\alpha)$  as testing threshold, has exact type-I error less than $\alpha$ by the bound  \eqref{propeq:testPhi}. Figure \ref{fig:procindtest} summarizes the procedure.
\begin{proposition}\label{prop:testPhi}{\sc(Type-I Error Bound.)}
Let $\alpha\in (0,1)$ and let a scoring function $\hat{s}=\mathcal{A}(\mathcal{D}^-_{n_-}, \mathcal{D}^+_{n_+})$. The test statistic for testing \eqref{eq:problem2}, based on the second part of the data $\mathcal{D}'^{-}_{n'_-}\cup\mathcal{D}'^{+}_{n'_+}$, is defined by: 
\begin{equation}\label{eq:test}
\Phi_{\alpha}^{\phi}=
%: \mathcal{D}_{n'}'(\hat{s}) \in (-\infty, \; + \infty ]^{n'} \mapsto 
\mathbb{I}\left\{ \frac{1}{n_+'}\widehat{W}^{\phi}_{n'_-,n'_+}(\hat{s}, \sigma')> \int_0^1\phi(u)\dd u  +  q^{\phi}_{n'_-,n'_+}(\alpha) \right\}
\end{equation}
   Under $\mathcal{H}_0$, we have for any pair of distributions $(H\otimes G, F)$ and for all $1\leq n_-'< n'$  and $1\leq \; n_+'< n'$:
   \begin{equation}\label{propeq:testPhi}
       \mathbb{P}_{\mathcal{H}_0}\left\{ \Phi_{\alpha}^{\phi}
(\mathcal{D}_{n'}'(\hat{s}))
=+ 1  \right\}\leq \alpha~,
   \end{equation}
   where $\mathcal{D}_{n'}'(s)$ denotes the dataset obtained by mapping the observations of $\mathcal{D}_{n'}'$ by any scoring function $s$.
\end{proposition}
The type-I error is exactly controlled, and essentially independent of the scoring function and holds true for any sample size $n'$. 
\begin{figure*}[t!]
    \fbox{\begin{minipage}{\textwidth} 
        \medskip
    \begin{center}
    {\large \textbf{Ranking-based Independence Rank Testing}}
    \end{center}
    \medskip
 \small
 \hspace{0.5cm}{\bf Input.} Collection of $N\geq 1$ $\iid$ copies $\mathcal{D}_N=\{(\bX_1,\bY_1)\; \ldots,\; (\bX_N,\bY_N)\}$ of $(\bX,\bY)$;  subsample sizes $n=n_++n_-<N$ and $n' = N-n=n'_++n'_-$; bipartite ranking $\mathcal{A}$ algorithm operating on the class $\S_0$ of scoring functions on $\mathcal{X}\times \mathcal{Y}$; score-generating function $\phi$; target level  $\alpha\in (0,1)$; quantile $q^{\phi}_{n'_-,n'_+}(\alpha)$.
    \begin{enumerate}		
        \item {\bf Splitting and  Shuffling.} Divide the initial sample into two subsamples $\mathcal{D}_N = \mathcal{D}_{n} \cup \mathcal{D}_{n'}'$.
        
        Independently from the $(\bX_i,\bY_i)'s$, draw uniformly at random two independent permutations $\sigma$ and $\sigma'$ in $\mathfrak{S}_{n_-}$ and $\mathfrak{S}_{n'_-}$ respectively, in order to build the independent samples: 
        $  \mathcal{D}^-_{n_-}=\{(\bX_i, \bY_{\sigma(i)})_{ 1\leq i\leq  n_-}\},\quad \mathcal{D}^+_{n_+}=\{(\bX_i, \bY_{i})_{ 1+n_-\leq i\leq  n}\}$, and 
        $\mathcal{D}'^{-}_{n'_-}=\{(\bX_i, \bY_{n+\sigma'(i-n)})_{ 1+n\leq i\leq n+ n'_-}\} , \quad \mathcal{D}'^{+}_{n'_+}=\{(\bX_i, \bY_{i})_{ 1+n+n'_-\leq i\leq  N}\}$.
            \item {\bf Bipartite Ranking.} 
   Run the bipartite ranking algorithm $\mathcal{A}$ based on the pooled training dataset $  \mathcal{D}_n=\mathcal{D}^-_{n_-}\cup \mathcal{D}^+_{n_+}$ %$  \mathcal{D}_n  = \mathcal{D}_{n_-}^- \cup   \mathcal{D}_{n_+}^+ $ 
   built at the previous step, in order to learn the scoring function $\hat{s}
            =\mathcal{A}(\mathcal{D}_n )$.
            \item {\bf Scoring and Two-sample Rank Statistic.} Build the univariate positive/negative subsamples using the scoring function $\hat{s}$ learned at the previous step   $\{\hat{s}(\bX_{n+1}, \bY_{n+\sigma'(1)}), \ldots,\hat{s}(\bX_{n+n'_-}, \bY_{n+\sigma'(n'_-)})\}$
                and $ \{\hat{s}(\bX_{n+n'_-+1},\bY_{n+n'_-+1}), \ldots, \hat{s}(\bX_N, \bY_N)\}$.
            Sort them by decreasing order to compute %By means of the scoring function $\hat{s}$ learned at the previous stage, build the univariate positive/negative subsamples
            %of magnitude 
            %so as to compute the statistic 
            \begin{equation}\label{eq:step2}
               \widehat{W}^{\phi}_{n'_-,n'_+}(\hat{s}, \sigma')=\sum_{i=1+n+n'_-}^N\phi\left( \frac{R'_{\sigma'}\left(\hat{s}(\bX_i,\bY_i)\right)}{n'+1} \right)~.
               \end{equation}
            \end{enumerate}
           \hspace{0.5cm}
          {\bf Output.} Compute the outcome of the test of level $\alpha$ based on the test statistic \eqref{eq:step2}, $\ie$, accept $\cH_0$ if:
      $$
      \frac{1}{n_+'}\widehat{W}^{\phi}_{n'_-,n'_+}(\hat{s}, \sigma')\leq \int_0^1\phi(u)\dd u  +  q^{\phi}_{n'_-,n'_+}(\alpha) ~, \quad \text{and reject it otherwise.}
      $$ 
    \end{minipage} }
    \caption{Ranking-based independence rank test.}\label{fig:procindtest}
\end{figure*}

\subsection{Nonasymptotic Theoretical Guarantees Under the Alternative - Error Bound}
We now investigate the theoretical properties of the test procedure previously described in the specific situation, where the bipartite ranking step is accomplished by maximizing, over a class $\S_0$ of scoring functions $s(x,y)$ on $\mathcal{X}\times \mathcal{Y}$, the empirical $W_{\phi}$-ranking performance measure computed from $\mathcal{D}^{-}_{n_-}\cup \mathcal{D}^{+}_{n_+}$:
\begin{eqnarray}\label{eq:Whatind2}
    \widehat{W}^{\phi}_{n_-,n_+}(s, \sigma)=\sum_{i=1+n_-}^n\phi\left( \frac{R_{\sigma}(s(\bX_i,\bY_i))}{n+1} \right)~,
\end{eqnarray}
where  $R_{\sigma}(t) = \sum_{i=1+n_-}^n \mathbb{I}\{s(\bX_i,\bY_i) \leq t\} + \sum_{i=1}^{n_-} \mathbb{I}\{s(\bX_i,\bY_{\sigma(i)}) \leq t\}$. We thus consider
\begin{equation}\label{eq:shatemp}
    \hat{s}\in \argmax_{s\in \S_0} \; \widehat{W}_{n_-,n_+}^{\phi}(s, \sigma)~.
\end{equation}
We  focus on establishing a uniform nonasymptotic bound for the type-II error of the test statistic $ \Phi_{\alpha}^{\phi}$. %The subsequent analysis 
It relies on the generalization properties of \eqref{eq:shatemp} $\wrt$ the deficit of $W_{\phi}$-ranking performance, investigated at length in \cite{CleLimVay21} (practical optimization issues are beyond the scope of the present paper, one may refer to \cite{CleLimVay21} for a dedicated study). 
The following technical assumptions are required to apply the related guarantees, and  refer to the Suppl.  Material for explicit definitions and details.
\begin{hyp}\label{hyp:phic2}
	The score-generating function $\phi : [0,1] \mapsto \RR$, is nondecreasing, of class $\mathcal{C}^2$.
\end{hyp}
\begin{hyp}\label{hyp:sabscont}
Let $M>0$.	For all $s\in \S_0$, the pushforward distributions of $F$ and $H\otimes G$ by the mapping $s(x,y)$ are continuous, with density functions that are twice differentiable and have Sobolev $\mathcal{W}^{2,\infty}$-norms %\footnote{Refer to the Suppl.  material for the definition of a Sobolev space.} 
 bounded by $M<+\infty$.
\end{hyp}
\begin{hyp}\label{hyp:VC}
	The class of scoring functions $\S_0$ is a Vapnik-Chervonenkis ({\sc VC}) class of finite {\sc VC} dimension $\V<\infty$.
\end{hyp}
Considering the quantity $W^*_{\phi}-\int_{0}^1\phi(u)\dd u$ to describe the departure from the null hypothesis $\mathcal{H}_0$ (see Theorem \ref{thm:rationale}) and the bias model $W^*_{\phi}-\sup_{s\in \S_0}W_{\phi}(s)$ inherent in the bipartite ranking step (when formulated as empirical $W_{\phi}$-ranking performance maximization), we introduce the two (nonparametric) classes of pairs of probability distributions on $\mathcal{X}\times \mathcal{Y}$.
\begin{definition}\label{def:depart} Let $\varepsilon>0$. %and $\S_0$ be a class of scoring functions 
	Denote by $\mathcal{H}_1(\varepsilon)$ the set of alternative hypotheses corresponding to all of probability distributions $F$ on $\mathcal{X}\times \mathcal{Y}$ s.t.
$ W^*_{\phi}-\int_{0}^1\phi(u)\dd u \geq \varepsilon~$,
where we recall $W_{\phi}^*= W_{\phi}(s^*) =  W^*_{\phi}(\dd F/\dd (H\otimes G) )$ for any $s^*\in \S^*$.
\end{definition}
\begin{definition}\label{def:bias}
Let $\delta>0$, $\S_0\subset \S$. We denote by $\mathcal{B}(\delta)$ the set of  all pairs $(H\otimes G, F)$  of probability distributions on $\mathcal{X}\times \mathcal{Y}$  such that
$ W^*_{\phi}- \sup_{s\in \S_0}W_{\phi}(s) \leq \delta$.
\end{definition}
The theorem below provides a rate bound for the type-II error of the ranking-based rank test \eqref{eq:test} of size $\alpha$. It depends on the sample sizes $n$ used for bipartite ranking, and on $n'=N-n$ for performing the rank test based on the learned scoring function, see Fig. \ref{fig:procindtest}. 

\begin{theorem}\label{thm:typeII} {\sc (Type-II error bound.)} Let $\phi(u)$ be a score-generating function and $\varepsilon>\delta>0$. Let $\sigma, \sigma'$ two independent permutations drawn $\resp$ from   $\mathfrak{S}_{n_-}$ and $\mathfrak{S}_{n'_-}$, independent of the $\bX_i$s, $\bY_j$s. Fix $\alpha\in (0,1)$. Suppose that Assumptions \ref{hyp:phic2}-\ref{hyp:VC} are fulfilled. Let $p \in (0,1)$ such that $n \wedge n' \geq 1/p$. Set $n_+= \lfloor p n \rfloor$ and $n_- = \lceil (1-p)n \rceil = n - n_+$, as well as  $n_+'= \lfloor pn' \rfloor$ and $n_-' = \lceil (1-p)n' \rceil = n' - n_+'$.
Then, there exist constants $C_1$ and $C_2 \geq 24$, depending on $(\phi, \; \V)$, such that the type-II error of  the test \eqref{eq:test} is uniformly bounded:
\begin{multline}\label{eq:typeII}
	\underset{\begin{subarray}{c}(H\otimes G, F)
 \in \mathcal{H}_1(\varepsilon)\cap \mathcal{B}(\delta)
 \end{subarray}}{\sup}
 \mathbb{P}_{\mathcal{H}_1}\left\{ \Phi^{\phi}_{\alpha} =0  \right\}  \leq 18 \exp\left(-\frac{Cn'( \varepsilon - \delta)^2}{16}\right)\\
	+ C_2 \left(1+\frac{\varepsilon-\delta}{32C_1\kappa_p}\right) ^{ -np\kappa_p( \varepsilon - \delta)/ (8C_2)}
\end{multline}
as soon as $n'\geq 4\log(18/\alpha)/(C(\varepsilon-\delta)^2)$ and $n\geq  16C_1^2/(p(\varepsilon-\delta)^2)$, with constants $\kappa_p = p\wedge  (1-p)$, $C=8^{-1}\min\left(p/\lVert \phi \rVert_{\infty}^2,  (p \lVert \phi' \rVert_{\infty}^2)^{-1}, ((1-p) \lVert \phi' \rVert_{\infty}^2)^{-1} \right) $,  the $C_j$'s are explicitly detailed in the proof. 
\end{theorem}
The first term results from the control of the type-II error of a univariate rank statistic. The second term  relies on Theorem 5 established in \cite{CleLimVay21}, inherited from the learning stage of the scoring function.  %({\it Step 2}). 
If the bias $\delta$ induced by the learning step is guaranteed to be smaller that the departure $\varepsilon$ from  $\mathcal{H}_0$, such that $\varepsilon - \delta>0$,  and if this quantity is kept fixed, then both terms in \eqref{eq:typeII} converge to zero when both %the sample sizes 
$n, \; n'\to\infty$. 
 Importantly, the error rate related to the hypothesis test is \emph{independent} on the dimensions of the spaces $\X$ and $\Y$. 
The only term dependent on those dimensions comes from the learning step, through the choice of bipartite ranking related to the class of scoring functions $\S_0$. Precisely, only the constants $C_1$ and $C_2$ depend on the dimensions of  $\X$ and $\Y$ as inherited by the {\sc VC} dimension $\V$ of $\S_0$. We illustrate this bound and its parameters ($\varepsilon, \delta$) in the context of Example \ref{exgauss} in the Suppl.  Material. \\
This result is important and new to the literature  for testing independence under nonparametric alternatives. It is, to the best of our knowledge, the first finite sample  probabilistic uniform control of the type-II error. The power of test statistics from the literature comparatively suffers from the underlying dimensions, see \cite{RamdasAAAI15}. The estimator of those statistics indeed take the form of $U$-statistics based on multivariate observations,  for which it has been proved to be subject to misspecification of the asymptotic distribution under nonparametric alternatives, see \cite{HXDGustat23}. Hence, our proposed method circumvents this limitation by computing the test statistic based on univariate samples that are mapped thanks to the scoring function solution of the bipartite ranking problem.

\section{NUMERICAL EXPERIMENTS}\label{sec:exp}
This section presents the empirical performance of our proposed method (Fig. \ref{fig:procindtest}), by illustrating the theoretical testing guarantees of section \ref{sec:main} through: 1) high-dimensional settings and non-monotonic class of alternatives, and 2) application to fair learning by testing for {\it statistical parity} based on real data published in \cite{jesus2022turning}. We mainly consider  synthetic datasets to exactly control the departure from independence.  We refer to the Supp. Material, Section \ref{appsec:expes} for details on the implementation and additional experiments.
These experiments can be reproduced using the Python code available  at {\small \url{https://github.com/MyrtoLimnios/independence_ranktest}}.

\paragraph{Ranking-based independence rank tests.} 
We implemented the Ranking Forest  algorithm (\texttt{rForest},  \cite{CDV13}) to solve \textit{Step 2}, following the empirical results   in \cite{CleLimVay21test}. 
We selected two score-generating functions to compute the rank statistic \eqref{eq:test} for \textit{Step 3}:
%$\widehat{W}^{\phi}_{n'_-,n'_+}(\hat{s})$: 
$\phi(u)=u$ (\texttt{rForest$_{MWW}$}, \cite{Wil45}) % recovering  the famous Mann-Whitney-Wilcoxon ranksum statistic (\cite{Wil45}), 
and $\phi(u)=u\mathbb{I}\{ u\geq u_0\}$ with  $u_0\in \{0.85, \; 0.90, \; 0.95\}$ (\texttt{rForest$_{u_0}$}, \cite{CV07})  considering only the $1-u_0$ higher ranks in the computation of the statistic corresponding to the beginning of the $\roc$ curve. 

\paragraph{Evaluation criteria and  experimental parameters.} Once all methods are calibrated for the range of significance levels $\alpha \in (0,1)$, we compare the graphs of the rate of rejecting $\mathcal{H}_0$ under $\mathcal{H}_1$, and also at fixed $\alpha=0.05$ exposed in tables in the Supp. Material. These criteria are computed over $B=100$ Monte-Carlo samplings, with $95\%$ confidence interval, and plotted against the  \textit{dependence} parameter $\rho \in \RR$, as function of the \textit{departure} level $\varepsilon \in(0,1)$, see Def. \ref{def:depart}.

\paragraph{Probabilistic model and experimental parameters.} 
We continue on  Ex. \ref{exgauss} motivated by the results in \cite{HXDGustat23} referred to as model (GL). 
Consider $(\bX, \; \bY) \sim \mathcal{N}(e_d,\Gamma_{\rho})$, where $e_{d}\in \mathbb{R}^d$ the null vector, $\text{Cov}(X^1,Y^k) = \rho$, for all $k\leq l$ and $\Gamma_{\rho,i,j} = \delta_{ij}$ otherwise.
We implement model (M1) for non-monotonic set of alternatives, wherein $X^1 = \rho \cos \Theta + \omega_1/4$, $Y^1 = \rho \sin \Theta + \omega_2/4$, with 
$\rho \in \{1, 2, 3\}$,
$\omega_i\sim \mathcal{N}(0,1)$, $i\in \{1,2\}$,  and $\Theta\sim\mathcal{U}([0,2\pi])$ all variables being independent, and with $d\in \{4, 10, 26\}$, $N\in \{500, 2000\}$. (M1) is extended for high dimension to both a sparse (M1s) and dense (M1d) models, see the Supp. Material, Section \ref{appsec:expes} therein.

The number of random permutations for our procedure is $K_p \in \{10,50\}$ under $\mathcal{H}_1$. The pooled sample size $N$ is fixed, with $n = 4N/5$ and $n' = N/5$, and set $q = l = d/2$. 

\paragraph{Benchmark tests.} We implement two state-of-the-art multivariate and nonparametric tests, 
namely the unbiased estimator of the Hilbert-Schmidt Independence Criterion (\texttt{HSIC}, \cite{Grett07ind}), with the recommended Gaussian kernel with bandwidth the median heuristic of the distance between the points in the merged sample ($\eg$ \cite{GBRSS12}), and the centered estimator of the Distance correlation computed with either the $L_1$ or the $L_2$ distances (\texttt{dCor$_{L_1}$}, \texttt{dCor$_{L_2}$}, \cite{Szekely07}). These methods require an additional implementation to estimate the null quantile, $\eg$ done by a permutation procedure. Due to their high computational complexity ($\mathcal{O}(N!)$), we restricted to a fixed number of permutations $K_0=200$. 

\paragraph{Results and discussion.}
We focus on the ability of the ranking-based method  to reject $\mathcal{H}_0$ for small dependence $\rho$ and for increasing  dimension $d$, depending on the choice of $\phi(u)$. First, the proposed method is distribution free under $\mathcal{H}_0$ for any bipartite ranking algorithm, hence its calibration
 only depends on $n_-', n_+'$, $\phi$ and $\alpha$. State-of-the-art (SoA) methods do not have this advantage in comparison. Other procedures than the implemented permutation-based one, approximate the asymptotic null distribution of the related statistics, namely using the Gamma distribution for the \texttt{HSIC}, see \cite{Grett07ind} Section 3. However, this method is proved to be  subject to misspecification under nonparametric assumptions, as proved in  \cite{HXDGustat23}, resulting in false estimation of the testing threshold and thus incorrect $p$-values. %, see Fig.s in the Suppl.  Material for the control of the empirical type-I error. 
Notice that, for the proposed ranking-based tests, the number of random permutations $K_p$ required to estimate the product of the marginal distributions, is lower than that for the estimation of the SoA's null threshold: we propose to only sample from both $\mathfrak{S}_{n_-}$ and $\mathfrak{S}_{n'_-}$, compared to  $\mathfrak{S}_{N}$. The experiments show that  %\ref{fig:expesgauss}, 
 for a well calibrated ranking-method, one achieves high empirical power with minimal number of permutations ($K_p\in\{10, 20,  50\}$), see Fig. \ref{tab:all500} and \ref{fig:expesGLd4}.  
For small sample sizes, \texttt{RTB} is not competitive as it has lower power for increasing $u_0$: fewer observations are considered and yielding larger variance for the estimation of the rank statistic. \texttt{RTB} has, however, experimentally showed higher accuracy for estimating the beginning of the true $\roc$ curve ($\roc^*$) in \cite{CleLimVay21}. \texttt{RTB}  also achieves competitive rejection rates to \texttt{MWW} for larger $N$,% (Fig. \ref{tab:M2_2000}),
 and to SoA for models (GL, M1), see Fig. \ref{tab:all500}.  %It also illustrates that 
When the  dimension increases $d\in \{N/10, N/5, N/2\}$,  fixing  $N$,  the performance of \texttt{MWW} remains high, $\eg$ Fig. \ref{tab:all500}. Notice that the data generating processes are designed not to suffer from signal-to-noise low ratio, for high dimension $d$ especially. However, there is a clear difference in the performances depending on the range of that ratio: the sparser the signal is and the smaller the rejection rates are for the SoA methods. For (GL), see Fig. \ref{fig:expesGLd4} ($d=4$) and  \ref{fig:expesGLd10} ($d=10$) especially, wherein the ratio equals to $\resp$ $1/4$ and $1/10$, \texttt{rForest} exhibits higher power for lower departures from $\cH_0$, see also (M1s) Fig. \ref{tab:M2_HDs} and plots \ref{fig:expesM1sd50}, \ref{fig:expesM1sd100}. On the contrary, for denser models, $\eg$ (M1d) Fig. \ref{tab:M2_HDd}, SoA methods have similar performances with \texttt{MWW}, however \texttt{RTB} shows no power for small departures $\rho$. The randomization related to \texttt{rForest} increases the chances to select important information, whereas for dense models, it might be ignored, see  \cite{clem13emptree} for further empirical analysis. 
Lastly,  both \texttt{HSIC} and \texttt{dCor$_{L2}$} show similar experimental performances as expected,  see  \cite{SejGrett13}.
To conclude, for all $\phi$ and $d$, the rejection rates of the ranking-based tests increase with the departure $\varepsilon$.  They empirically outperform the comparative SoA tests, studied for  non-monotonic and sparse high dimensional models.

\begin{figure}[ht!]
    %
   %\hspace{-1cm}
    %  \centering
      \begin{tabular}{ccc}
        %\hspace{-0.7cm}
    \parbox{0.33\textwidth}{
    \includegraphics[scale=0.35]{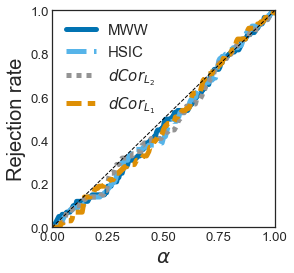}
    \subcaption{(GL), $\rho = 0.0, d=4$}
    } %\hspace{1cm}
    
    \parbox{0.30\textwidth}{
    \includegraphics[scale=0.35]{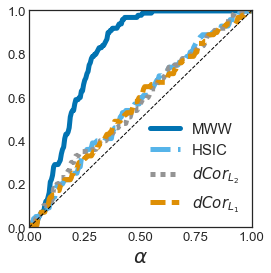}
    \subcaption{$\rho = 0.2, d=4$}
    \includegraphics[scale=0.35]{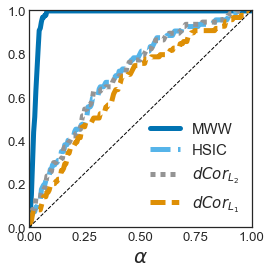}
    \subcaption{(GL), $\rho=0.4, d=4$}
    }
    \parbox{0.30\textwidth}{
    \includegraphics[scale=0.35]{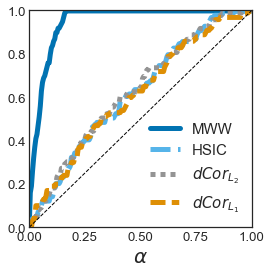}
    \subcaption{(GL), $\rho = 0.3, d=4$}
    \includegraphics[scale=0.35]{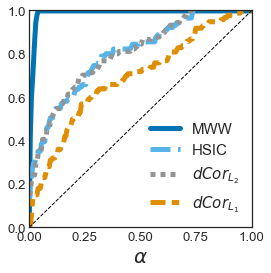}
    \subcaption{(GL), $\rho=0.5, d=4$}
    }
    \end{tabular}
    \caption{Plots of the  rejection rate under $\cH_0$ (a) and $\cH_1$ (b-e) against the significance level $\alpha\in(0,1)$ for (GL) with $\phi(u)=u$ (\texttt{rForest$_{MWW}$}), $\rho = 0.0$ (a)  $\rho = 0.2$ (b), $\rho = 0.3$ (c), $\rho = 0.4$ (d), $\rho = 0.5$ (e). The parameters are fixed to  $N=1000$, $d=4$, $K_p=10$, $K_0=200$, $B=100$ for all experiments.}
    \label{fig:expesGLd4}
    \end{figure}

    \begin{figure}[ht!]
      %
     %\hspace{-1cm}
       \centering
        \begin{tabular}{cccc}
          \hspace{-0.7cm}
      \parbox{0.24\textwidth}{
      \includegraphics[scale=0.3]{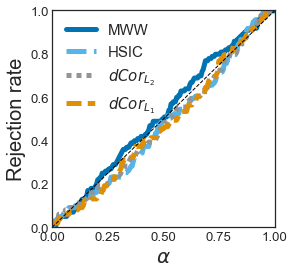}
      \subcaption{(GL),\\ $\rho = 0.0, d=10$}
      } 
      \parbox{0.23\textwidth}{
      \includegraphics[scale=0.3]{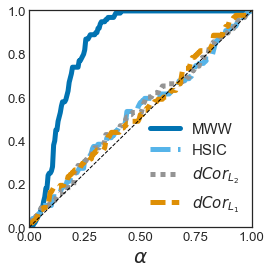}
      \subcaption{(GL),\\ $\rho = 0.10, d=10$}
      }
      \parbox{0.23\textwidth}{
      \includegraphics[scale=0.3]{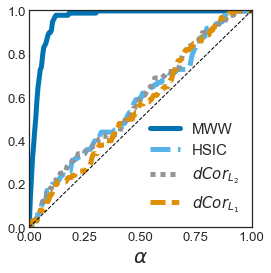}
      \subcaption{(GL),\\ $\rho = 0.15, d=10$}
      }
      \parbox{0.23\textwidth}{
      \includegraphics[scale=0.3]{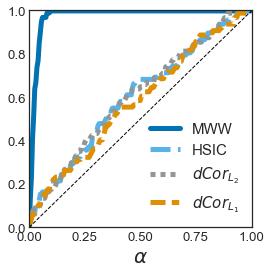}
      \subcaption{(GL),\\  $\rho=0.20, d=10$}
      }
  
      \end{tabular}
      \caption{Plots of the  rejection rate under $\cH_0$ (a) and $\cH_1$ (b-d) against the significance level $\alpha\in(0,1)$ for (GL) with $\phi(u)=u$ (\texttt{rForest$_{MWW}$}), $\rho = 0.0$ (a)  $\rho = 0.10$ (b), $\rho = 0.15$ (c), $\rho = 0.20$ (d). The parameters are fixed to  $N=1000$, $d=10$, $K_p=10$, $K_0=200$, $B=100$ for all experiments.}
      \label{fig:expesGLd10}
      \end{figure}

\paragraph{Interpretation of the null assumption rejection.} We recall that certain bipartite ranking algorithms, such as those proposed in \cite{CDV09} or \cite{CDV13}, produce scoring functions that can be interpreted to a certain extent. As explained in section 5 of \cite{CDV09}, the \textit{relative importance} of each component of the argument $(X,Y)$ of a scoring function $s(x,y)$ defined by a `ranking tree' (or by a `forest of ranking trees') can be easily quantified. When applied to the testing problem considered here, this interpretability tool may permit to identify the components mainly responsible for the departure from the independence assumption (or equivalently the departure of the $\roc$ curve from the diagonal) possibly assessed from the data by means of the methodology we promote. We further refer to similar discussions on interpretability of the learned decision rule in the context of two-sample testing, when formulated as a classification learning problem in $\eg$ \cite{lopez17iclr,kubler22pmlr}.

\paragraph{Real data experiment: testing for statistical parity.}
 In the context of `responsible' statistical prediction, a significant number of  works have studied {\it fair} statistical methods, aiming to be unbiased/fair {\it wrt.} {\it protected} attributes/ subgroups considered as sensitive. In particular,
 {\it Statistical parity} is achieved when a decision rule producing a set of outcomes $\bX$ based on an ensemble of covariates $\bZ$, is independent of a set of protected attributes $\bY$. 
 We propose to test for statistical parity formulated as a test for independence between $\bX$ and $\bY$ as in \eqref{eq:problem}.
 If $\bX$ is univariate and discrete, typical methods in fairness propose to learn a classification model to predict $\bX$, wherein both $(\bX, \; \bZ)$ are used, and then to measure or test for statistical independence between the predicted $\bX$ and the protected attributes $\bY$, see $\eg$ \cite{fermanianFair}. 
 We propose to apply our proposed method in that context, to assess whether a typical algorithm learns to predict the outcome under statistical parity, when  both outcomes $\bX$ and protected variables $\bY$ are continuous and valued in spaces of possibly dimensions $q,l>1$. We use the synthetic Bank Account Fraud (BAF) dataset developed by \cite{jesus2022turning}, and generated from real datasets of  frauds in anonymized  bank account openings.  
 BAF has $31$ explanatory variables plus one indicating the possible occurrence of fraud. It has unbalanced representation of frauds and all features can be modeled as continuous observations. 
 We selected three potentially protected variables related to the personal identity of the clients, namely the age of the client (\texttt{Age}), an indicator level of similarity between  the name of the client and personal email address (\texttt{Name}), and the number of emails received for applicants with same date of birth four weeks prior to fraud (\texttt{Date}). 
 We gather the distributions of the empirical $p$-values in Fig. \ref{fig:fraud}, based on a $5$-fold cross-validation. For each fold, a Random Forest algorithm is trained to predict the probability of \texttt{Fraud} $\bX$, and our ranking-based procedure (Fig. \ref{fig:procindtest}) is used  to estimate the associate $p$-value of \eqref{eq:problem}. We subsampled at random from the original data set $N=10^3$ while keeping the proportion of \texttt{Fraud} from the original dataset fixed. This plot shows that we cannot reject at level $\alpha = 0.05$ the statistical independence between the predicted probability of fraud and the protected variables $\bY$. 

 \begin{figure}[ht!]
    %\vspace{-0.3cm}
      \centering
    \includegraphics[scale=0.45]{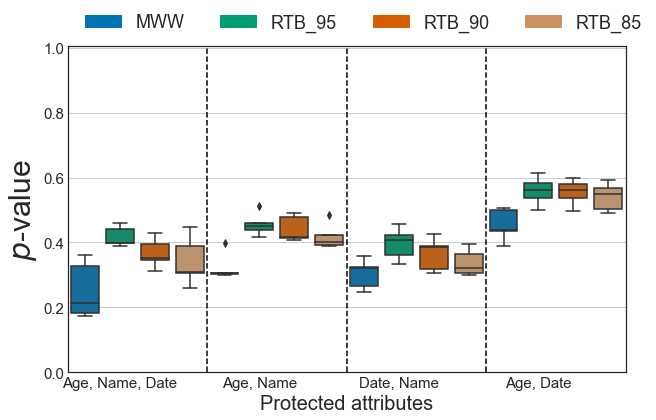}
    \caption{Boxplots of the $p$-values for different sets of protected attributes. The experimental parameters are fixed to $N=10^3$, $K_p =10$, $q=1, l\in \{2,3\}$, $5$-fold cross-validation, $31$ features, based on the open-source dataset available \cite{jesus2022turning}.} 
    \label{fig:fraud}
\end{figure}

\section{CONCLUSION}\label{sec:conclusion}
We have proposed a novel approach, involving a preliminary bipartite ranking stage, to test independence between random variables in a nonparametric and possibly high-dimensional setting. Nonasymptotic error bounds have been established for this method, and its theoretical optimality properties are confirmed by numerical experiments, showing that it generally detects small departures from the independence much better than its competitors and  resists to the high dimension especially in sparse settings.

\subsubsection*{Acknowledgements}
We  thank the reviewers for the useful comments. Myrto Limnios was supported by Novo Nordisk Foundation Grant NNF20OC0062897.

\vskip 0.2in
\bibliography{Ref_independence}
\bibliographystyle{plain}

\appendix

\newpage
%\section{Supplementary Material}

This material supplements the article {\it On Ranking-based Tests of Independence}. Section \ref{appsec:preliminaries} gathers additional properties on $\roc$ analysis,  and on the assumptions required to prove the main Theorem 2, as well as developing Example 1. In Section \ref{appsec:proofs}, we derive the detailed proofs for all theoretical results stated in the main corpus of the article. Finally, section \ref{appsec:expes} presents additional numerical experiments on synthetic data. Importantly, we  prove the nonasymptotic control of both type-I and type-II errors formulated as concentration inequalities and show empirical evidence for the competitiveness of our proposed method. %Section \ref{appsec:expes} gathers experimental studies in the continuity of section 4, with comparison to state-of-the-art methods.
We first recall the proposed ranking-based rank test of independence procedure in Figure \ref{appfig:proc} for the sake of clarity.

\begin{figure}[ht!]
  \fbox{\begin{minipage}{\textwidth} 
      \medskip
  \begin{center}
  {\large \textbf{Ranking-based Independence Testing}}
  \end{center}
  \medskip

\hspace{0.5cm}{\bf Input.} Collection of $N\geq 1$ $\iid$ copies $\mathcal{D}_N=\{(\bX_1,\bY_1)\; \ldots,\; (\bX_N,\bY_N)\}$ of $(\bX,\bY)$;  subsample sizes $n=n_++n_-<N$ and $n' = N-n=n'_++n'_-$; bipartite ranking $\mathcal{A}$ algorithm operating on the class $\S$ of scoring functions on $\mathcal{X}\times \mathcal{Y}$; score-generating function $\phi$; target level  $\alpha\in (0,1)$; quantile $q^{\phi}_{n'_-,n'_+}(\alpha)$.
      %; class of permutations $\mathfrak{S}_{n_-}$ and $\mathfrak{S}_{n'_-}$.\\

      \medskip
  \begin{enumerate}		
     % \item[]
       
      \item {\bf Splitting and  Shuffling.} Divide the initial sample into two subsamples $\mathcal{D}_N = \mathcal{D}_{n} \cup \mathcal{D}_{n'}'$.
      
      Independently of the $(\bX_i,\bY_i)'s$, draw uniformly at random two independent permutations $\sigma$ and $\sigma'$ in $\mathfrak{S}_{n_-}$ and $\mathfrak{S}_{n'_-}$ respectively, in order to build the independent samples: 
      $  \mathcal{D}^-_{n_-}=\{(\bX_i, \bY_{\sigma(i)})_{ 1\leq i\leq  n_-}\},\quad \mathcal{D}^+_{n_+}=\{(\bX_i, \bY_{i})_{ 1+n_-\leq i\leq  n}\}$, and 
      $\mathcal{D}'^{-}_{n'_-}=\{(\bX_i, \bY_{n+\sigma'(i-n)})_{ 1+n\leq i\leq n+ n'_-}\} , \quad \mathcal{D}'^{+}_{n'_+}=\{(\bX_i, \bY_{i})_{ 1+n+n'_-\leq i\leq  N}\}$.
      
          \item {\bf Bipartite Ranking.} 
 Run the bipartite ranking algorithm $\mathcal{A}$ based on the pooled training dataset $  \mathcal{D}_n=\mathcal{D}^-_{n_-}\cup \mathcal{D}^+_{n_+}$ %$  \mathcal{D}_n  = \mathcal{D}_{n_-}^- \cup   \mathcal{D}_{n_+}^+ $ 
 built at the previous step, in order to learn the scoring function 
          \begin{equation}\label{eq:step1}
          \hat{s}
          =\mathcal{A}(\mathcal{D}_n )~.
          \end{equation}
          
          \item {\bf Scoring and Two-sample Rank Statistic.} Build the univariate positive/negative subsamples using the scoring function $\hat{s}$ learned at the previous step   $\{\hat{s}(\bX_{n+1}, \bY_{n+\sigma'(1)}), \ldots,\hat{s}(\bX_{n+n'_-}, \bY_{n+\sigma'(n'_-)})\}$
              and $ \{\hat{s}(\bX_{n+n'_-+1},\bY_{n+n'_-+1}), \ldots, \hat{s}(\bX_N, \bY_N)\}$.
          Sort them by decreasing order to compute %By means of the scoring function $\hat{s}$ learned at the previous stage, build the univariate positive/negative subsamples
          
          %of magnitude 
          %so as to compute the statistic 
          \begin{equation}\label{supeq:step2}
             \widehat{W}^{\phi}_{n'_-,n'_+}(\hat{s}, \sigma')=\sum_{i=1+n+n'_-}^N\phi\left( \frac{R'_{\sigma'}\left(\hat{s}(\bX_i,\bY_i)\right)}{n'+1} \right)~,
             \end{equation}
  where  $$R'_{\sigma'}(t) = \sum_{i=1+n+n'_-}^N \mathbb{I}\{\hat{s}(\bX_i,\bY_i) \leq t\}+ \sum_{i=1+n}^{n+n'_-} \mathbb{I}\{\hat{s}(\bX_i,\bY_{n+\sigma'(i-n)}) \leq t\}~.$$
          \end{enumerate}
         \hspace{0.5cm} {\bf Output.} Compute the outcome of the test of level $\alpha$ based on the test statistic \eqref{eq:step2}: accept the hypothesis $\cH_0$ of independence if:
    $$
    \frac{1}{n'_+}\widehat{W}^{\phi}_{n'_-,n'_+}(\hat{s}, \sigma')\leq \int_0^1\phi(u)\dd u  +  q^{\phi}_{n'_-,n'_+}(\alpha) ~, \quad \text{and reject it otherwise.}
    $$ 
  \end{minipage} }
  \caption{Ranking-based independence testing.}\label{appfig:proc}
\end{figure}

\section{PRELIMINARIES}\label{appsec:preliminaries}

This subsection gathers additional definitions and properties important to the main corpus. We first expose results related to $\roc$ analysis. Then, we provide complementary material to Assumptions (1-3) required to prove the main Theorem 2, deriving the nonasymptotic uniform bound of the type-II error of the test statistic based on Eq. \eqref{eq:step2}. We consider  same notations as in the main corpus of the article. 

\subsection{{\rm ROC} analysis}
\begin{lemma} (\cite{CV09ieee}) Let $\bZ$ denote either $\bX_+$ or $\bX_-$, and define the likelihood ratio $\Psi(z)=\dd F_{+} / \dd F_{-}(z) $.  The property below holds true a.s. 
  \begin{equation*}
  \Psi(\bZ)=\frac{\dd F_{\Psi, +}}{\dd F_{\Psi,-}}(\Psi(\bZ))~,
  \end{equation*}
where $F_{\Psi, +}$ ($\resp$ $F_{\Psi, -}$) is the pushforward distribution $F_+$ ($\resp$ $F_{-}$) by the likelihood ratio.
\end{lemma}

\begin{proposition}(\cite{CV09ieee})\label{prop:pteROC} For any probability distributions $ F_{+} $ and $ F_{-} $, and any scoring function $s:\mathcal{Z}\to\RR$, the following assertions hold true.
\begin{enumerate}
  \item[(i)] $\roc(s, 0)=0$ and  $\roc(s, 1)=1$.
   \item[(ii)] The $\roc$ curve is invariant by any  nondecreasing transform   $c: \RR\to \RR$ of a scoring function $s(z)$ on $(0,1)$: $\roc(c\circ s, \cdot) = \roc(s, \cdot)$.
    \item[(iii)] Let a scoring function $s(z)$. Suppose both distributions $ F_{+} $ and $ F_{-} $ are continuous. Then, the associated $\roc$ curve of the function $s(z)$ is  differentiable $\IFF$ the pushforward distributions $F_{s,+}$ and $F_{s,-}$ are continuous. 
\end{enumerate}
\end{proposition}
\subsection{Sobolev and \vc-classes of functions}

\paragraph{Sobolev space of functions.} Assumption 2 requires that for all $s\in \S_0$, the pushforward distributions of $F$ and $H\otimes G$ by the mapping $s(x,y)$ are continuous, with density functions that are twice differentiable and have Sobolev $\mathcal{W}^{2,\infty}$-norms 
bounded by  a finite constant $M>0$. 

We recall that the Sobolev space $\mathcal{W}^{2,\infty}$ is composed of  all Borelian functions $f:\mathbb{R}\rightarrow \mathbb{R}$, such that $f$ and its first and second order weak derivatives $f'$ and $f''$ are bounded almost-everywhere. It is a Banach space when equipped with the norm $\vert\vert f\vert\vert_{2,\infty}=\max\{\vert\vert f\vert\vert_{\infty},\; \vert\vert f'\vert\vert_{\infty},\; \vert\vert f''\vert\vert_{\infty}   \}$, where $\vert\vert.\vert\vert_{\infty}$ is the norm of the Lebesgue space $L_{\infty}$ of Borelian and essentially bounded functions.

\paragraph{\vc-type classes of functions.} We recall below the definition of \vc-type class of functions formulated in Assumption 3. We further refer to \cite{vdVWell96}, Chapter 2.6. therein, for additional  generalizations, details and examples.
\begin{definition}\label{def:VCclass}
A class $\mathcal{F}$ of real-valued functions defined on a measurable space $\Z$ is a bounded \vc-type class with parameter $(A,\V)\in (0,\; +\infty)^2$ and constant envelope $L_{\mathcal{F}}>0$ if for all $\varepsilon\in (0,1)$:
\begin{equation}\label{eq:covnumb}
\underset{Q}{\sup} \; N(\mathcal{F},L_2(Q),\varepsilon L_{\mathcal{F}})\leq \left(\frac{A}{\varepsilon} \right)^{\V}~,
\end{equation}
where the supremum is taken over all probability measures $Q$ on $\Z$ and the smallest number of $L_2(Q)$-balls of radius less than $\varepsilon$ required to cover class $\mathcal{F}$ ($\ie$ covering number) is meant by $N(\mathcal{F},L_2(Q),\varepsilon)$.
\end{definition}

In particular, a bounded {\sc VC} class of functions with finite {\sc VC} dimension $V$ is of {\sc VC}-type, with  $\V=2(V-1)$ and $A=(cV(16e)^V)^{1/(2(V-1))}$, where $c$ is a universal constant, see $\eg$ \cite{vdVWell96}, Theorem 2.6.7 therein.

\subsection{Multivariate Gaussian framework - Example  1 continued}\label{appssec:gaussexample}
This section extends Example 1, $\ie$,  for testing independence between two  multivariate  Gaussian $\rv$. 
We focus on deriving the explicit constants appearing in the bound that are related to: (i) the testing problem through the departure from the null $\varepsilon>0$, and the bias $\delta>0$, and (ii) the complexity of the selected class of scoring functions $\S_0$ (Assumption 3).

\paragraph{Framework and procedure.}
Consider a centered Gaussian $\rv$ $(\bX,\; \bY)$ with definite positive covariance $\Gamma$, valued in $\mathbb{R}^{q+l}$. Denote by $\Gamma_{\bX}$ and $\Gamma_{\bY}$ the (definite positive) covariance matrices of the components $\bX$ and $\bY$. The oracle class of scoring functions $\S^*$ is composed of the increasing transforms of the likelihood ratio, taking the form of the quadratic scoring function:
\begin{equation*}
s : z\in \mathbb{R}^{q+l}\mapsto   z^t (\Gamma^{-1}-\text{diag}(\Gamma_{\bX}^{-1}, \Gamma_{\bY}^{-1}) )z~.
\end{equation*}
and define $\theta^* = \Gamma^{-1}-\text{diag}(\Gamma_{\bX}^{-1}, \Gamma_{\bY}^{-1})$. Following the procedure summarized in Figure \ref{appfig:proc},  we thus propose to solve {\it Step 2} by learning the optimal scoring function in the class:
$$ \S_0(\Theta) = \{  s_{\theta}:z\in \mathbb{R}^{q+l}\mapsto   z^t \theta z, \quad \theta \in \Theta\}~,$$
where $\Theta$ is a subset of real definite positive matrices of size $\mathbb{R}^{(q+l)\times (q+l)}$. 
Notice that, for any $\rv$ $\bZ$ drawn either from $H\otimes G$ or $F$, the $\rv$ $s_{\theta}(\bZ)$ for any $\theta \in \Theta$, being  a quadratic transform of multivariate Gaussian $\rv$, is a weighted sum of $\chi^2(1)$ $\rv$.   

\paragraph{{\sc VC} dimension of $\S_0(\Theta)$.} We analyze the {\sc VC} dimension of the class $\S_0(\Theta)$ to obtain explicit relations of the constants appearing in Theorem 2 with the dimensions of the spaces $\X$ and $\Y$. Notice that, 
\begin{equation*}
  \S_0(\Theta) =  \{  s_{\theta}:z\in \mathbb{R}^{q+l}\mapsto \langle \theta,   z  z^t\rangle_F, \quad \theta \in \Theta\} 
\end{equation*}
where $\langle\cdot,\cdot\rangle_F$ is the Frobenius inner product in $\mathbb{R}^{(q+l)\times (q+l)}$. This yields  the collection of subgraphs taking the form:
\begin{equation*}
  \{\{ (z, t) \in \RR^{(q+l)}\times \RR \mapsto \langle \theta,   z  z^t\rangle_F > t\}, \quad \theta \in \Theta\}~. 
\end{equation*}
It is a \vc-class of functions by  recognizing linear separator for matrix networks (taking the sign function), where $\theta$ is definite positive thus of full rank, the {\sc VC} dimension can be  upperbounded by $c(q+l)^2$, with $c=2\log(24)$ constant. We refer to \cite{KhavariTensor21}, Theorem 2 therein, stating general upperbounds applied to tensor networks. Therefore, the constants involved in Definition \ref{def:VCclass} for the class $\S_0$ are: $A = (cV(16e)^V)^{1/(2(V-1))}$, with $V = c(q+l)^2 $. 
Applying the permanence properties proved in \cite{CleLimVay21}, all resulting classes of functions implied in the analysis of the $R$-statistic in {\it Step 2} are therefore {\sc VC} bounded and of parameters depending similarly to those of the basis class $\S_0$, see Lemma 14,19,20 in particular.
By Proposition 2.1 \cite{GineGuill01}, we can see that the dominant constant $C_1$ appearing in Theorem 2, as function of the parameters of the class $\S_0$, is a linear combination of $V$ and $V^2$, while $C_2$ is a linear combination of $V$ and $\sqrt{V}$.

\paragraph{Definition 1: interpretation of the alternative hypothesis $\cH_1$.} 
Notice that for any distribution $H\otimes G$ and $F$, $\ie$ not necessarily Gaussian, choosing the score-generating function $\phi(u)=u$ trivially leads to $\cH_1(\varepsilon):\quad \auc (s) - 1/2 \geq \varepsilon/(1-p)$. The deviation from the null hypothesis thus depends linearly on $\varepsilon$.

\paragraph{Definition 2: bipartite ranking bias.} In this setting $\delta=0$ as $\S_0(\Theta)\subset \S^*$.

\subsection{Nonlinear Dependence - Example 2}\label{exgumbel} 
\cite{Gumbel60} proposed a construction of dependent  absolute continuous univariate $\rv$ that allows for larger class of alternative hypotheses. Let $X, \; Y$ of $\resp$ distribution functions $h(\dd x)$ and $g(\dd y)$, the class of joint distributions indexed by the dependence parameter $\rho\in[-1,1]$ can be defined by $
  f_{\rho}(x,y) = h(x)g(y)(1+\rho (2H(x)-1)(2G(y)-1))$, 
  yielding the explicit oracle class $\mathcal{S}^*$ by noticing that % of the form 
  $ \Psi_{\rho} (x, y) = \rho (2H(x)-1)(2G(y)-1)$.

\section{TECHNICAL PROOFS}\label{appsec:proofs}
\subsection{Proof of Theorem  1}%\ref{thm:rationale}.
The equivalence between assertions $(i)$ and $(ii)$ results from Corollary 7 in \cite{CDV09}, applied to the pair $(H\otimes G, \; F)$ and combined with the equality $\roc^*(\cdot)=\roc(\Psi,\; \cdot)=\roc(s^*,\; \cdot)$ for any $s^* \in \S^*$ by Proposition \ref{prop:pteROC}. %stating that  $\Psi (\bX,\bY) = \frac{\dd F }{ \dd H\otimes G } (\bX,\bY)  = \frac{\dd F_{\Psi} }{ \dd (H\otimes G)_{\Psi} }    (\Psi(\bX,\bY) )$ for $(\bX,\bY)$ drawn either from $F$ or $ H\otimes G$.
One establishes the  remaining equivalences by using Equation (4). 
\subsection{Proof of Proposition 1 }%\ref{prop:build_samples}}
Assertions $(i)$ and $(ii)$ are inherent to the construction of the subsamples by mutual independence of all random pairs $(\bX_i, \; \bY_i)$s, and their independence with the permutations $\sigma, \; \sigma'$ independently drawn at random from $\mathfrak{S}_{n_-}$ and $\mathfrak{S}_{n'_-}$.
\subsection{Proof of Proposition 2 }%\ref{prop:build_samples}}
Let $\alpha\in (0,1)$,  $1\leq n_-'< n'$  and $1\leq \; n_+'< n'$. 
Consider a scoring function $\hat{s}=\mathcal{A}(\mathcal{D}^-_{n_-}, \mathcal{D}^+_{n_+})$ solution of \textit{Step 2}, see Fig. \ref{appfig:proc}. By Proposition 1, $\hat{s}$ is independent of both $\mathcal{D}_{n_-'}^{'-}$ and $ \mathcal{D}_{n_+'}^{'+}$,  hence conditioning on the subsample $\mathcal{D}^-_{n_-} \cup \mathcal{D}^+_{n_+}$  under the null hypothesis yields  \textit{a.s.}:
\begin{multline*}
\mathbb{P}_{\mathcal{H}_0}\left\{ \Phi_{\alpha}^{\phi}=+ 1  \bigm\vert  \mathcal{D}^-_{n_-} \cup  \mathcal{D}^+_{n_+}\right\} 
= \mathbb{P}_{\mathcal{H}_0}\left\{\frac{1}{n_+'}\widehat{W}^{\phi}_{n_-',n_+'}(\widehat{s}, \sigma')>\int_0^1\phi(u)\dd u \right.\\
\left. +q^{\phi}_{n'_-,n'_+}(\alpha)   \bigm\vert  \mathcal{D}^-_{n_-} \cup  \mathcal{D}^+_{n_+}\right\} \leq \alpha~,
\end{multline*}
where the first equality holds true by definition of the test statistic. The inequality results from the definition of the $(1-\alpha)$-quantile $q^{\phi}_{n'_-,n'_+}(\alpha)$ of the pushforward distribution of $\mathcal{L}^{\phi}_{n'_-,n'_+}$, by the mapping $w\mapsto (1/n')w-\int_0^1\phi(u)\dd u$,  depending only on $\phi$, $n'_+$ and $n'_-$. Then taking the expectation $\wrt$ $\mathcal{D}^-_{n_-} \cup  \mathcal{D}^+_{n_+}$  concludes the proof.
\subsection{Proof of Theorem 2}\label{supsec:proofth} %\ref{thm:typeII}.
Let $\alpha\in(0,1),\;  \varepsilon>0,\;  \delta>0$, and consider a scoring function $\hat{s}=\mathcal{A}(\mathcal{D}^-_{n_-}, \mathcal{D}^+_{n_+})\in\S_0$, solution of the  bipartite ranking step  (\textit{Step 2}) when formulated as the maximization of the empirical $W_{\phi}$-performance criterion over the class $\mathcal{S}_0$, see Fig. \ref{appfig:proc}. Observe that for all  %$(F_+, \; F_-)\in \mathcal{H}_1(\varepsilon)$, we have: %
alternatives $(H\otimes G,\; F)$ in  $\mathcal{H}_1(\varepsilon) \cap  \mathcal{B}_1(\delta)$, the deviation of the rank statistic from the null decomposes \textit{a.s.} as:
\begin{multline}\label{appeq:decompdevi}
  \frac{1}{n_+'}\widehat{W}^{\phi}_{n_-',n_+'}(\widehat{s}, \sigma')-\int_{0}^1\phi(u)\dd u = \left\{ \frac{1}{n_+'}\widehat{W}^{\phi}_{n_-',n_+'}(\widehat{s}, \sigma') - W_{\phi}(\widehat{s})  \right\} \\
  - \left\{  W_{\phi}^* -  W_{\phi}(\widehat{s})\right\} +
  \left\{ W_{\phi}^*-\int_{0}^1\phi(u)\dd u \right\}~,
\end{multline}
and the generalization deviation of the $W_{\phi}$-performance criterion satisfies, by Definition 2: 
\begin{equation}\label{appeq:gen}
W_{\phi}^* -  W_{\phi}(\widehat{s})\leq 2 \sup_{s\in \S_0}\left\vert \frac{1}{n_+} \widehat{W}_{n_-,n_+}(s, \sigma)-W_{\phi}(s)\right\vert + \delta~.
\end{equation}
We can bound the type-II error on the samples $\mathcal{D}^{'-}_{n_-'}\cup\mathcal{D}^{'+}_{n_+'}$ as follows:
%The test statistic  defined in Eq. \eqref{eq:test}  yields
\begin{multline}\label{eq:dec2}
\mathbb{P}_{H,G}\left\{ \Phi^{\phi}_{\alpha} =0  \right\}
= \mathbb{P}_{H,G}\left\{\frac{1}{n_+'}\widehat{W}^{\phi}_{n_-',n_+'}(\widehat{s}, \sigma')-\int_{0}^1\phi(u)\dd u \leq q_{n_-',n_+'}^{\phi}(\alpha)   \right\} \\
\leq \mathbb{P}_{H,G}\left\{  2 \sup_{s\in \S_0}\left\vert \frac{1}{n_+} \widehat{W}_{n_-,n_+}(s, \sigma)-W_{\phi}(s)\right\vert + \left\vert \frac{1}{n_+'} \widehat{W}^{\phi}_{n_-',n_+'}(\widehat{s}, \sigma')-W_{\phi}(\widehat{s}) \right\vert \right.\\
 \left. \geq 
\varepsilon - \delta- \sqrt{\frac{\log(18/\alpha)}{Cn'}} \right\} 
\end{multline}
where $C=8^{-1}\min\left(p/\lVert \phi \rVert_{\infty}^2,  (p \lVert \phi' \rVert_{\infty}^2)^{-1}, ((1-p) \lVert \phi' \rVert_{\infty}^2)^{-1} \right) $ and as soon as $n'\geq 4\log(18/\alpha)/(C(\varepsilon-\delta)^2)$. We sequentially used Eq. \eqref{appeq:decompdevi} and \eqref{appeq:gen},  and Proposition \ref{appprop:rate_quantile} to upperbound the quantile applied to samples of sizes  $n_+', \; n_-'$, proved in section  \ref{appaddproofs}. 

We now apply Theorem 5 in \cite{CleLimVay21} to bound the uniform deviation of the $W_{\phi}$-ranking performance criterion to its estimator based on the two-samples $\mathcal{D}^-_{n_-} \cup  \mathcal{D}^+_{n_+}$, %with threshold  $t=(\varepsilon-\delta)/4$,  %=  C_4(p\wedge  (1-p))$ 
such that for all $n \geq 16 C_1^2/ (p (\varepsilon-\delta)^2)$:
\begin{multline}\label{eq:2part}
\mathbb{P}_{H,G}\left\{  2 \sup_{s\in \S_0}\left\vert\frac{1}{n_+}\widehat{W}_{n_-,n_+}(s, \sigma)-W_{\phi}(s)\right\vert   \geq  \frac{\varepsilon - \delta}{2}  \right\} \\
\leq C_2
\exp\left\{ -\frac{np(p\wedge  (1-p))}{4C_2}( \varepsilon - \delta) \log \left(1+\frac{\varepsilon-\delta}{16C_1(p\wedge  (1-p))}\right) \right\}~,
\end{multline}
as soon as $n  \geq 16C_1^2/(p(\varepsilon - \delta)^2)$, constants $C_1>0, \; C_2 \geq 24$ depend on $\phi, \; \V$ %,   $C_4\geq C_1$  depends only on $\phi$ and $C_3=\log(1+C_4/(4C_2))/(C_2C_4)$ 
of values detailed in the dedicated proof, see \cite{CleLimVay21}, Appendix section B.3 therein. 

We can now upperbound the deviation of the two-sample rank statistic $\wrt$ the $W_{\phi}$-ranking performance criterion by conditioning on the first subsample $\mathcal{D}_n= \mathcal{D}^-_{n_-}\cup \mathcal{D}^+_{n_+}$ and applying the inequality \eqref{appeq:bound1Drstat},  to the two independent samples:
$$ 
\{\hat{s}(\bX_{n+1}, \bY_{n+\sigma'(1)}), \; \ldots,\; \hat{s}(\bX_{n+n'_-}, \bY_{n+\sigma'(n'_-)})\}$$
              $$ \cup \{\hat{s}(\bX_{n+n'_-+1},\bY_{n+n'_-+1}), \; \ldots, \; \hat{s}(\bX_N, \bY_N)\}
$$
\begin{multline}\label{eq:1part}
 \mathbb{P}_{H,G}\left\{  \left\vert \frac{1}{n_+'} \widehat{W}^{\phi}_{n_-',n_+'}(\widehat{s}, \sigma')-W_{\phi}(\widehat{s}) \right\vert   \geq \frac{\varepsilon - \delta}{2}  -  \sqrt{\frac{\log(18/\alpha)}{Cn'}} \bigm\vert \mathcal{D}^-_{n_-}\cup \mathcal{D}^+_{n_+}  \right\} \\
 \leq 18 \exp\left\{-\frac{Cn'\left( \varepsilon - \delta\right)^2}{4}\right\}~.
\end{multline}

Finally, we obtain the desired bound by taking the expectation on the last inequality \eqref{eq:1part} and combining it with Eq. \eqref{eq:2part} using the union bound.

\subsection{A nonasymptotic inequality for the testing threshold}\label{appaddproofs}

%\paragraph{.}  
Let $\{X_{\varepsilon,1},\;\ldots,\; X_{\varepsilon,n_{\varepsilon}}\}$ with $\varepsilon \in \{-, +\}$, be two independent $\iid$ random samples, drawn from univariate probability distributions $F_{\varepsilon}$. Recall that the univariate two-sample linear rank statistic based on these samples is defined by
\begin{equation}\label{appeq:crit_emp}
  \hat{W}^{\phi}_{n_-,n_+}=\sum_{i=1}^{n_+}\phi\left( \frac{R(X_{+,i})}{n+1} \right)~,
\end{equation}
where  the ranks $R(X_{+,i}) = \sum_{\epsilon\in\{-,+\}}\sum_{j=1}^{n_{\epsilon}}\mathbb{I}\{s(X_{\epsilon, j}) \leq X_{+,i}  \}$, for all $i\leq n_+$
The proposed class of linear rank statistics is distribution-free under the null, hence allows for the exact computation of the testing threshold for any sample sizes.  Proposition \ref{appprop:rate_quantile} provides an upperbound  for the $(1-\alpha)$-quantile $q^{\phi}_{n_-,n_+}(\alpha)$ of the pushforward distribution of $\mathcal{L}^{\phi}_{n_-,n_+}$ by the mapping $w\mapsto (1/n)w-\int_0^1\phi(u)\dd u$. It proves to be of order $\mathcal{O}_{\mathbb{P}}(n^{-1/2})$ and only depending  on $\phi$, $n_+, \; n_-$ and $\alpha$.

\begin{proposition}\label{appprop:rate_quantile} Let $p\in (0,1)$ and $n\geq 1/p$.  Let the score-generating function $\phi(u)$ satisfy Assumption 1. Set $n_+ = \lfloor pn \rfloor$ and $n_- = \lceil (1-p)n \rceil = n - n_+$. Then, for any $\alpha\in (0,1)$, the $(1-\alpha)$-quantile  satisfies  a.s.:
\begin{equation}
q^{\phi}_{n_-,n_+}(\alpha) \leq \sqrt{\frac{\log(18/\alpha)}{Cn}}~,
\end{equation}
where 
$C=8^{-1}\min\left(p/\lVert \phi \rVert_{\infty}^2,  (p \lVert \phi' \rVert_{\infty}^2)^{-1}, ((1-p) \lVert \phi' \rVert_{\infty}^2)^{-1} \right) $.
\end{proposition}
\begin{proof} The proof relies on the concentration results established in \cite{CleLimVay21}, see Theorem 5 in particular, and builds upon the linearization technique exposed therein. Define by $F = pF_+ + (1-p)F_-$ the mixture $\cdf$ of the pooled sample and of empirical estimator $\widehat{F}_{n} (t)= (1/n)\sum_{\varepsilon \in \{ +,- \}}\sum_{i \leq n_{\varepsilon} }\mathbb{I}\{ X_{\varepsilon,i} \leq t\}$.  By considering $\phi(u)$ satisfying Assumption 1, writing its Taylor expansion of order $2$ evaluated at $ n\widehat{F}_n(X_{+,i} )/(n+1)$ around $\bF(X_{+,i})$ for $1\leq i\leq n_+$, and  summed over $i\leq n_+$, results in a \textit{a.s.} decomposition of the statistic Eq. \eqref{eq:crit_emp}. We refer to Eq. (B.3,4) in \cite{CleLimVay21} for the detailed arguments.

The  terms of the resulting expansion of order one are  composed of two $U$-statistics, for which the Hoeffding decomposition results in the linearization below: 
\begin{multline}\label{eq:linear}
  \frac{1}{n_+}\widehat{W}_{n_-,n_+}^{\phi} - W_{\phi} = \widehat{W}_{\phi}-  W_{\phi} +\frac{1}{n_+} \left(\widehat{V}_{n_+}^+ -\mathbb{E}\left[ \widehat{V}_{n_+}^+ \right]\right)
  +\frac{1}{n_+}\left( \widehat{V}_{n_-}^- -\mathbb{E}\left[  \widehat{V}_{n_-}^- \right]\right)\\
+\frac{1}{n_+} \mathcal{R}_{n_-,n_+}~,
\end{multline}

where:
\begin{eqnarray*}
W_{\phi}&=& \mathbb{E}[\left(\phi \circ \bF\right)(X_{+})]~, \\
\widehat{W}_{\phi}&=& \frac{1}{n_+} \sum_{i=1}^{n_+} \left(\phi \circ \bF\right)(X_{+,i})~, \\
\widehat{V}_{n_+}^+ &=& \frac{n_+-1}{n+1} \sum_{i=1}^{n_+} \int_{X_{+,i}}^{+\infty} (\phi' \circ \bF)(u) \dd F_+(u)~,\\
\widehat{V}_{n_-}^- &=& \frac{n_+}{n+1} \sum_{i=1}^{n_-} \int_{X_{-,i}}^{+\infty}( \phi' \circ \bF)(u) \dd F_+(u)~,
\end{eqnarray*}
and $ \mathcal{R}_{n_-,n_+}$ is the sum of the Taylor-Lagrange residual term $\widehat{T}_{n_-,n_+}$, and of the terms of  order at most $\mathcal{O}_{\mathbb{P}}(n^{-1})$ inherited from the (two) Hoeffding decompositions. Precisely, it inherits from the linear statistics of order $\mathcal{O}_{\mathbb{P}}(n^{-1})$ defined by $ \widehat{R}_{n_-,n_+} $, and both remainder terms being degenerate $U$-statistics. % resulting from the Hoeffding decomposition of both $U$-statistics (one-sample $U_{n_+}$ and two-sample $U_{n_-, n_+}$).  
We detail hereafter the main steps for obtaining a nonasymptotic exponential deviation bound of the univariate rank statistic $(1/n_+)\widehat{W}_{n_-,n_+}^{\phi}$ based on Eq. \eqref{eq:linear}. Following \cite{CleLimVay21}, define  the (nonsymmetric) bounded kernels defined on $\mathbb{R}^2$ by:%$\resp$  defined on  $\mathcal{X}\times \mathcal{X}$ and $\mathcal{X}\times \mathcal{Y}$,  by:
\begin{eqnarray*}
       k (z,z') &=& \mathbb{I}\{z'\leq z\} (\phi' \circ \bF)(z)~.%\\
   %\ell (x,y)& =&\mathbb{I}\{y\leq x\} ( \phi' \circ \bF)(x)~.
\end{eqnarray*} 
Then  
\begin{equation*}
\mathcal{R}_{n_-,n_+} = \widehat{R}_{n_-,n_+} + \frac{n_+(n_+-1)}{n+1}  U_{n_+} (k)+\frac{n_+n_-}{n+1}  U_{n_-,n_+} (k)+  \widehat{T}_{n_-,n_+}~, %U_{n_-,n_+} (\ell)+  \widehat{T}_{n_-,n_+}~,
\end{equation*}
where $U_{n_+}(k) $ is the one-sample degenerate $U$-statistic of order $2$ based on the positive sample with kernel $k$, $ U_{n_-, n_+}(k) $ %$ U_{n_-, n_+}(\ell) $
is the two-sample degenerate $U$-statistic  of degree $(1,1)$ based on the two samples $\{X_{\varepsilon,1},\;\ldots,\; X_{\varepsilon,n_{\varepsilon}}\}$, with $\varepsilon \in \{-, +\}$, with kernel $k$. %$\ell$. 

Noticing that 
\begin{equation*}
\vert	\mathcal{R}_{n_-,n_+}\vert \leq \vert \widehat{R}_{n_-,n_+}\vert +p^2n \vert U_{n_+}(k) \vert+ p(1-p)n\vert U_{n_-,n_+}(k) \vert + \vert \widehat{T}_{n_-,n_+}\vert~,
\end{equation*}
one can sequentially upperbounded the tail of each  term with threshold $t/16$, for any $t>0$,  in probability using: Hoeffding's classic exponential bound from \cite{Hoeffding63} with the union bound to $\widehat{R}_{n_-,n_+}$, Lemma 3 in \cite{NP87} applied to $U_{n_+} $, Lemma 27 in \cite{CleLimVay21} to $U_{n_+,n_-} $, and finally for $\widehat{T}_{n_-,n_+}$, one has:
\begin{multline*}
\frac{1}{n_+}\vert \widehat{T}_{n_-,n_+} \vert 
\leq  \Vert \phi''\Vert_{\infty} \left( \sup_{t \in\mathbb{R} }  \left(\widehat{F}_{n}(t) -\bF(t) \right)^2 
+  \frac{ 1}{(n+1)^2}\right)\\
\leq 3  p^2 \Vert \phi''\Vert_{\infty} \sup_{t \in\mathbb{R} }  \left(\widehat{F}_{n_+}(t) -F_+(t) \right)^2  \\
+ 3 (1-p)^2 \Vert \phi''\Vert_{\infty} \sup_{t \in\mathbb{R} }  \left(\widehat{F}_{n_-}(t) -F_-(t) \right)^2   + \frac{13\Vert \phi''\Vert_{\infty}}{n^2}~.
\end{multline*}
It remains to apply Dvoretzky–Kiefer–Wolfowitz inequality to each of the two first terms on the right hand side, while the third is negligeable $\wrt$ the others. 
This concludes to, for all  $nt \geq 512  \lVert \phi' \rVert_{\infty}^2/( p  \lVert \phi'' \rVert_{\infty} ) $:
\begin{equation}
\mathbb{P}\left\{ \vert \mathcal{R}_{n,m} \vert > \frac{t}{4}  \right\} \leq 12 \exp\left\{- \frac{Nt}{48\kappa_p \lVert \phi'' \rVert_{\infty}}\right\}~,
\end{equation}
and otherwise
\begin{equation}
  \mathbb{P}\left\{ \vert \mathcal{R}_{n_-,n_+} \vert > \frac{t}{4}   \right\} \leq 12 \exp\left\{- \frac{\alpha_p n^2t^2}{512 \lVert \phi' \rVert_{\infty}^2}\right\}~,
\end{equation}
where $\alpha_p = \min(p ,1-p)/(4(1-p))$, $\kappa_p = \max(p, 1-p)$.

It remains to apply Hoeffding exponential inequality to the other terms of the decomposition Eq. \eqref{eq:linear} with threshold $t/4$ as follows:

\begin{eqnarray*}
\mathbb{P}\left\{ \vert  	\widehat{W}_{\phi} -  W_{\phi} \vert >  \frac{t}{4}  \right\} &\leq & 2 \exp\left\{-\frac{pnt^2}{8\lVert \phi \rVert_{\infty}^2}\right\},\\
\mathbb{P}\left\{ \frac{1}{n_+} \left| \widehat{V}_{n_+}^+ -  \EE\left[\widehat{V}_{n_+}^+\right] \right| >  \frac{t}{4}   \right\} &\leq & 2 \exp\left\{-\frac{nt^2}{8p\lVert \phi' \rVert_{\infty}^2}\right\},\\
\mathbb{P}\left\{ \frac{1}{n_+} \left|  \widehat{V}_{n_-}^- - \EE\left[\widehat{V}_{n_-}^  -\right] \right| >  \frac{t}{4}   \right\} &\leq & 2\exp\left\{-\frac{nt^2}{8(1-p)\lVert \phi' \rVert_{\infty}^2}\right\}~.
\end{eqnarray*}

By virtue of the union bound, we obtain

\begin{equation}\label{appeq:bound1Drstat}
\mathbb{P}\left\{ \left\vert   \frac{1}{n_+} \widehat{W}_{n_-,n_+}^{\phi} -W_{\phi} \right\vert  >t \right\}\leq 18 \exp\{-Cnt^2\}~,
\end{equation}
where  $C=8^{-1}\min\left(p/\lVert \phi \rVert_{\infty}^2,(p \lVert \phi' \rVert_{\infty}^2)^{-1},((1-p) \lVert \phi' \rVert_{\infty}^2)^{-1} \right) $, concluding the proof. 

\end{proof}

\section{Additional Numerical Experiments}\label{appsec:expes}

This section details the technicalities related to the experiments exposed in the main corpus, as well as additional experiments on synthetic data. 

\paragraph{Experimental parameters.}  All results are shown with $95\%$ confidence interval based on $B\in \NN^*$ Monte-Carlo samplings. The number of random permutations for the benchmark tests is chosen so that the test is calibrated  $K_0 = 200$, % the tests fixed to $B_p =1000$, 
the number of random permutations for our proposed procedure is fixed to $K_p \in \{10, 20,  50\}$. The significance level is chosen equal to  $\alpha = 0.05$. We consider the pooled sample size $N\in\{500, 1000, 2000\}$, with $n = 4N/5$ and $n' = N/5$, where the subsamples are balanced $n_-=n_+=n/2$, $n_-'=n_+'=n'/2$, and denote by $d = 2q = 2l $. We choose the RTB parameter $u_0\in \{0.85,  \; 0.90, \; 0.95\}$.

\paragraph{Probabilistic models.}
We first consider different types of independence according to the following models. Define $\bX = (X^1, X^2, \ldots, X^{q})$ and $\bY=(Y^1, Y^2, \ldots, Y^{l})$, the first two models sample $\bX$ and $\bY$ according to the multivariate Gaussian distribution, in the continuity of Example \ref{exgauss}.

\begin{enumerate}
  \item[](GL) $(\bX, \; \bY) \sim \mathcal{N}(e_d,(1/\sqrt{d})\times \Gamma_{\rho})$, where $e_{d}\in \mathbb{R}^d$ the null vector, $\text{Cov}(X^1,Y^k) = \rho$, for all $k\leq l$ and $\Gamma_{\rho,i,j} = \delta_{ij}$ otherwise, $d\in \{4,10, 26,50\}$ for $N=500$ and $d\in \{4,10\}$ for $N=1000$.
 \item[](GL+) Covariance matrix from model (GL) extended for higher dimensions with $\text{Cov}(X^u,Y^k) = \rho$, for all $k\leq l$ and a $u\leq q$ only, and with $d \in \{100,250, 500\}$, $N=500$.

\end{enumerate}
Also, for (GL), the range of the dependence parameter $\rho$ are chosen such that the resulting $\Gamma_{\rho}$ is positive definite to show directional dependency. 
The following data generation distributions model non-monotonic alternative hypothesis. The first subset of coordinates $X^u, Y^v$'s are drawn according to the models below, and $X^i, \; Y^j$, for all $i,j\geq u, v$ are independently drawn from the Univariate distribution on $[0,1]$ and are independent of the first coordinate. %We fix $q=l=5$ for models (M1-2).

\begin{enumerate}
  \item[](M1) $X^1 = \rho \cos \Theta + \omega_1/4$, $Y^1 = \rho \sin \Theta + \omega_2/4$, with 
  $\rho \in \{1, 2, 3\}$,
  $\omega_i\sim \mathcal{N}(0,1)$, $i\in \{1,2\}$,  and $\Theta\sim\mathcal{U}([0,2\pi])$ all variables being independent, and with $d\in \{4, 10, 26\}$, $N\in \{500, 2000\}$.
  \item[](M1s) Sparse covariance matrix from model (M1)  extended for higher dimensions by generating the $X^u,  Y^v$'s, for $u,v\leq q/2, l/2$ according to (M1) and the $X^u,  Y^v$, for $u,v>q/2, l/2$ are drawn from the Univariate distribution on $[0,1]$, with $d\in \{100,250,500\}$, $N = 500$.
  \item[](M1d) Dense covariance matrix from model (M1)  extended for higher dimensions with by generating the $X^u,  Y^v$, for all coordinates $u,v\leq q, l$ according to (M1), with $d\in \{100,250,500\}$, $N = 500$. 
\end{enumerate}

Model (M1) was proposed for both the univariate and bivariate settings by \cite{BerrSam19} and further studied by \cite{AlbLMM22} and for very small sample sizes. 
We compare our results for models (M1s) and (M1d) to the benchmark tests to see the resistance to high dimension $d$.

  \begin{figure}[ht!]
    %
   %\hspace{-1cm}
    %  \centering
      \begin{tabular}{ccccc}
        %\hspace{-0.7cm}
    \parbox{0.25\textwidth}{
    \includegraphics[scale=0.3]{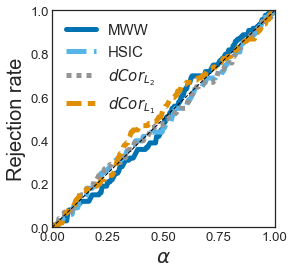}
    \subcaption{$\rho = 0.0, d=50$}
    } %\hspace{1cm}
    \parbox{0.23\textwidth}{
    \includegraphics[scale=0.3]{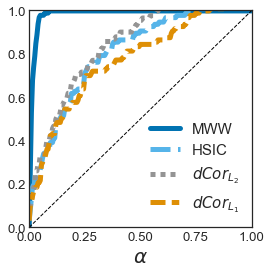}
    \subcaption{(M1s),\\ $\rho = 0.30, d=50$}
    \includegraphics[scale=0.3]{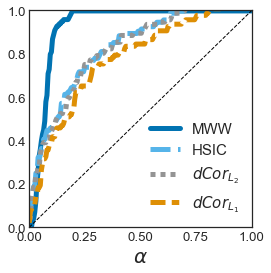}
    \subcaption{(M1d),\\ $\rho = 0.20, d=50$}
    }
    \parbox{0.23\textwidth}{
    \includegraphics[scale=0.3]{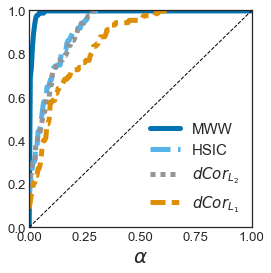}
    \subcaption{(M1s),\\ $\rho = 0.40, d=50$}
    \includegraphics[scale=0.3]{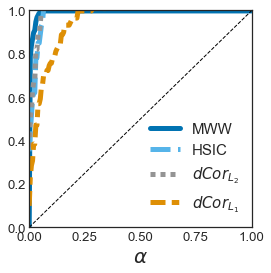}
    \subcaption{(M1d),\\ $\rho = 0.30, d=50$}
    }
    \parbox{0.23\textwidth}{
    \includegraphics[scale=0.3]{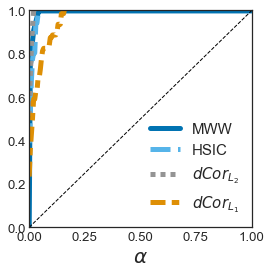}
    \subcaption{(M1s),\\ $\rho = 0.50, d=50$}
    \includegraphics[scale=0.3]{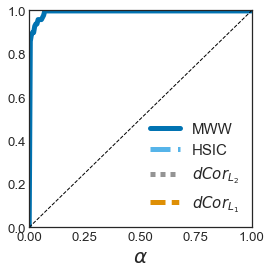}
    \subcaption{(M1d),\\ $\rho = 0.40, d=50$}
    }
  
    \end{tabular}
    \caption{Plots of the  rejection rate under $\cH_0$ (a) and $\cH_1$ (b-g) against the significance level $\alpha\in(0,1)$ for (M1s) top row, and (M1d) bottom row, with $\phi(u)=u$ (\texttt{rForest$_{MWW}$}), $\rho = 0.0$ (a)  $\rho \in (0.20, 0.50)$ (b-g). The parameters are fixed to  $N=500$, $d=50$, $K_p=10$, $K_0=200$, $B=100$ for all experiments.}
    \label{fig:expesM1sd50}
    \end{figure}
  
    \begin{figure}[ht!]
      %
     %\hspace{-1cm}
      %  \centering
        \begin{tabular}{ccccc}
          %\hspace{-0.7cm}
      \parbox{0.25\textwidth}{
      \includegraphics[scale=0.3]{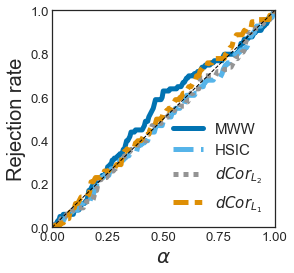}
      \subcaption{$\rho = 0.0, d=100$}
      } %\hspace{1cm}
      \parbox{0.23\textwidth}{
      \includegraphics[scale=0.3]{pval_test_M2_HDsNd50050_1001200_tree11_30__scale_03071239.png}
      \subcaption{(M1s),\\ $\rho = 0.30, d=100$}
      \includegraphics[scale=0.3]{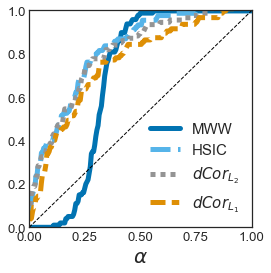}
      \subcaption{(M1d),\\ $\rho = 0.15, d=100$}
      }
      \parbox{0.23\textwidth}{
      \includegraphics[scale=0.3]{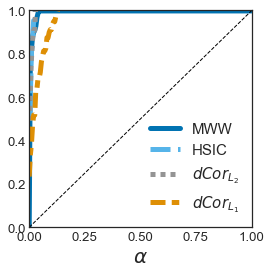}
      \subcaption{(M1s),\\ $\rho = 0.40, d=100$}
      \includegraphics[scale=0.3]{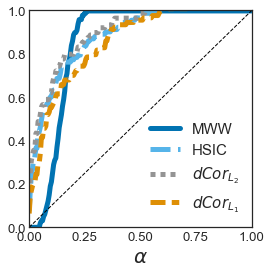}
      \subcaption{(M1d),\\ $\rho = 0.20, d=100$}
      }
      \parbox{0.23\textwidth}{
      \includegraphics[scale=0.3]{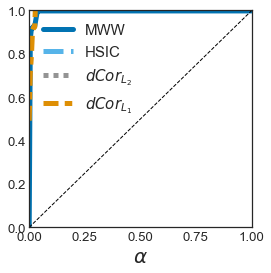}
      \subcaption{(M1s),\\ $\rho = 0.50, d=100$}
      \includegraphics[scale=0.3]{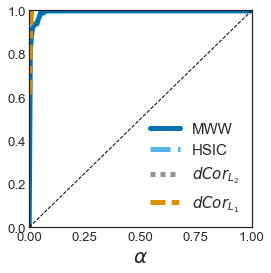}
      \subcaption{(M1d),\\ $\rho = 0.30, d=100$}
      }
    
      \end{tabular}
      \caption{Plots of the  rejection rate under $\cH_0$ (a) and $\cH_1$ (b-g) against the significance level $\alpha\in(0,1)$ for (M1s) top row, and (M1d) bottom row, with $\phi(u)=u$ (\texttt{rForest$_{MWW}$}), $\rho = 0.0$ (a)  $\rho \in (0.15, 0.50)$ (b-g). The parameters are fixed to  $N=500$, $d=100$, $K_p=10$, $K_0=200$, $B=100$ for all experiments.}
      \label{fig:expesM1sd100}
      \end{figure}

\begin{landscape}
\begin{table*}[ht!]
	%\small
	%\renewcommand{\arraystretch}{1.1}
	\centering
	\resizebox{2\textwidth}{!}{	
		\begin{tabular}{|l||c|ccc|c|ccc|c|ccc|}
			\hline
            \textbf{Model (GL)}  &\multicolumn{4}{c|}{\textbf{$N=500$, $d = 4$}   } & \multicolumn{4}{c|}{\textbf{$N=500$, $d=10$}  }& \multicolumn{4}{c|}{\textbf{$N=500$, $d=26$}  }\\% & \multicolumn{4}{c|}{\textbf{$N=500$, $d=50$}  }  \\
            \hline
            \textbf{Rejection rate of the null} & \text{$\cH_0$} & \multicolumn{3}{c|}{\textbf{$\cH_1$}}& \text{$\cH_0$} & \multicolumn{3}{c|}{\textbf{$\cH_1$}}&\text{$\cH_0$} & \multicolumn{3}{c|}{\textbf{$\cH_1$}} \\
          \textbf{Method }&  $\rho=0.0$ & $\rho=0.1$  & $\rho=0.3$ & $\rho=0.6$ &   $\rho=0.0$ & $\rho=0.05$ & $\rho=0.1$ & $\rho=0.15$ &  $\rho=0.0$ & $\rho=0.02$ & $\rho=0.05$ & $\rho=0.07$\\%$\rho=0.1$ & $\rho=0.3$ & $\rho=0.6$ \\ %& $\rho=0.0$ & $\rho=0.1$ &  $\rho=0.3$ & $\rho=0.6$  \\
			\hline 
			\hline 
			\bf rForest$_{MWW}$ &0.04 $\pm$ 0.19  &\bf0.47 $\pm$  0.50  &\bf0.92 $\pm$ 0.27 & \bf1.00 $\pm$  0.00 & 0.05 $\pm$ 0.22&\bf0.85 $\pm$ 0.36 &\bf0.92 $\pm$ 0.27&\bf0.98 $\pm$ 0.14& 0.04 $\pm$ 0.19  &\bf 0.98 $\pm$ 0.14 &\bf 0.99 $\pm$ 0.10 &\bf1.00 $\pm$ 0.00\\%1.00 $\pm$ 0.00   & 0.99 $\pm$ 0.10 &0.98 $\pm$ 0.14 \\%&  0.04 $\pm$ 0.19 & no &run with scaled gamma &run \\
      rForest$_{95}$& 0.01 $\pm$ 0.10 &0.02 $\pm$ 0.14 & 0.08 $\pm$ 0.27 &0.17 $\pm$ 0.35 &0.01 $\pm$ 0.10&0.04 $\pm$ 0.19&0.17 $\pm$ 0.38&0.16 $\pm$ 0.37& 0.01 $\pm$ 0.10 &  0.29 $\pm$ 0.46 & 0.42 $\pm$ 0.59 & 0.39 $\pm$ 0.49 \\% 0.43 $\pm$ 0.50 &0.31 $\pm$ 0.47& 0.30 $\pm$ 0.46 \\%& 0.01 $\pm$ 0.10\\
      rForest$_{90}$ & 0.02 $\pm$ 0.14&0.10 $\pm$ 0.30 & 0.23 $\pm$ 0.42 & 0.49 $\pm$ 0.50 &0.01 $\pm$ 0.10&0.32 $\pm$ 0.47&0.37 $\pm$ 0.46&0.46  $\pm$  0.50 & 0.01 $\pm$ 0.10& 0.64 $\pm$ 0.48 & 0.60 $\pm$ 0.49 & 0.77 $\pm$ 0.42\\%0.80 $\pm$ 0.40 &0.70 $\pm$ 0.46& 0.57 $\pm$ 0.50 \\%& 0.03 $\pm$ 0.17\\
      rForest$_{85}$ & 0.02 $\pm$ 0.14& 0.17 $\pm$ 0.38 & 0.35 $\pm$ 0.48 &0.71 $\pm$ 0.46 &0.01 $\pm$ 0.10&0.41 $\pm$ 0.49&0.52 $\pm$ 0.50& 0.63  $\pm$ 0.49 &0.02 $\pm$ 0.14&  0.79 $\pm$ 0.41 & 0.78 $\pm$ 0.42 & 0.87 $\pm$ 0.34 \\%0.89 $\pm$ 0.31 & 0.87 $\pm$ 0.38& 0.69 $\pm$ 0.46 \\%& 0.03 $\pm$ 0.17
            
			\hline
			HSIC& 0.06 $\pm$ 0.24 & 0.09 $\pm$ 0.29 & 0.06 $\pm$ 0.24 & 0.14 $\pm$ 0.35 & 0.06 $\pm$ 0.24 & 0.06 $\pm$ 0.24  & 0.09 $\pm$ 0.29 & 0.04 $\pm$ 0.19 & 0.06 $\pm$ 0.24 &0.03 $\pm$ 0.17 &0.06 $\pm$ 0.24 & 0.03 $\pm$ 0.17 \\% 0.09 $\pm$ 0.29 & 0.07 $\pm$ 0.26 \\ %& 0.04 $\pm$ 0.19 & 0.06 $\pm$ 0.24 & 0.06 $\pm$ 0.24 & 0.06 $\pm$ 0.24  \\
      dCor$_{L2}$ & 0.10 $\pm$ 0.30 & 0.06 $\pm$ 0.24 & 0.03 $\pm$ 0.17 & 0.12 $\pm$  0.33 &  0.07 $\pm$ 0.26 &0.03 $\pm$ 0.17 & 0.10 $\pm$ 0.30 & 0.11 $\pm$ 0.31 & 0.05 $\pm$ 0.22 &0.03 $\pm$ 0.17 &0.06 $\pm$ 0.24 &0.10 $\pm$ 0.30\\

      dCor$_{L1}$& 0.08 $\pm$ 0.27 & 0.06 $\pm$ 0.24 & 0.09 $\pm$ 0.29 & 0.15 $\pm$ 0.36 &0.05 $\pm$ 0.22& 0.08 $\pm$ 0.27&0.11 $\pm$ 0.31 & 0.09 $\pm$ 0.29& 0.04 $\pm$ 0.19 & 0.04 $\pm$ 0.20 & 0.04 $\pm$ 0.20 &0.0è $\pm$ 0.26 \\%0.06 $\pm$ 0.24 &  0.08 $\pm$ 0.27 & 0.06 $\pm$ 0.24 \\% &  0.03 $\pm$ 0.17& 0.07 $\pm$ 0.26 & 0.10 $\pm$ 0.30 & 0.04 $\pm$ 0.19 \\

			\hline
			\hline\hline
            \textbf{Model (M1)}  &\multicolumn{4}{c|}{\textbf{$N=500$, $d = 4$}   } & \multicolumn{4}{c|}{\textbf{$N=500$, $d=10$}  }& \multicolumn{4}{c|}{\textbf{$N=500$, $d=26$}  }  \\
            \hline
            \textbf{Rejection rate of the null} & \text{$\cH_0$} & \multicolumn{3}{c|}{\textbf{$\cH_1$}}& \text{$\cH_0$} & \multicolumn{3}{c|}{\textbf{$\cH_1$}}&\text{$\cH_0$} & \multicolumn{3}{c|}{\textbf{$\cH_1$}} \\
          \textbf{Method }&  $\rho=0.0$ & $\rho=1$ & $\rho=2$ & $\rho=3$ & $\rho=0.0$ & $\rho=1$ & $\rho=2$ & $\rho=3$  & $\rho=0.0$ & $\rho=1$ & $\rho=2$ & $\rho=3$ \\
			\hline 
			\hline 
			\bf rForest$_{MWW}$ & 0.04 $\pm$ 0.19  &\bf 0.78$\pm$ 0.42  &\bf 0.97$\pm$ 0.17  &\bf 0.99 $\pm$ 0.10& 0.04 $\pm$ 0.19 &\bf1.00 $\pm$ 0.00 &\bf1.00 $\pm$ 0.00&\bf1.00 $\pm$ 0.00& 0.04 $\pm$ 0.20 &\bf0.99 $\pm$ 0.10&\bf1.00 $\pm$ 0.00&\bf1.00 $\pm$ 0.00\\%1.00 $\pm$ 0.00 &  0.97 $\pm$ 0.17 &  0.97 $\pm$ 0.17 & 0.04 $\pm$ 0.20 & 1.00 $\pm$ 0.00 &1.00 $\pm$ 0.00 &1.00 $\pm$ 0.00 
       rForest$_{95}$ &  0.01 $\pm$ 0.10& 0.02$\pm$ 0.14 & 0.07 $\pm$ 0.26& 0.13$\pm$ 0.34 & 0.00 $\pm$ 0.00&0.16 $\pm$ 0.37 &0.70 $\pm$ 0.46&0.88 $\pm$ 0.33& 0.00  $\pm$ 0.00 &0.33 $\pm$ 0.47 &0.92 $\pm$ 0.27&0.99 $\pm$ 0.10 \\% 0.29 $\pm$ 0.46 & 0.28 $\pm$ 0.45 & 0.31 $\pm$ 0.46 & 0.00  $\pm$ 0.00 & 0.47  $\pm$ 0.50 & 0.50  $\pm$ 0.50 & 0.47  $\pm$ 0.50\\
       rForest$_{90}$  & 0.02 $\pm$ 0.14 & 0.16$\pm$ 0.37 &0.38 $\pm$ 0.47 & 0.52 $\pm$ 0.50 & 0.02 $\pm$ 0.14&0.67 $\pm$ 0.47 &0.98 $\pm$ 0.14 &\bf1.00 $\pm$ 0.00&   0.01 $\pm$ 0.10&0.80 $\pm$ 0.040&\bf1.00 $\pm$ 0.00&\bf1.00 $\pm$ 0.00\\% 0.64 $\pm$ 0.48 & 0.63 $\pm$ 0.49 & 0.66 $\pm$ 0.48 &    0.01 $\pm$ 0.10& 0.88 $\pm$ 0.33 & 0.86  $\pm$ 0.35  & 0.89  $\pm$ 0.31\\
       rForest$_{85}$  & 0.00 $\pm$ 0.00 & 0.23 $\pm$ 0.42  &0.56 $\pm$ 0.50 &  0.71 $\pm$ 0.46 &  0.01 $\pm$ 0.10& 0.88 $\pm$ 0.33&1.00 $\pm$ 0.00 &\bf1.00 $\pm$ 0.00&  0.01 $\pm$ 0.10& 0.89$\pm$ 0.31&\bf1.00 $\pm$ 0.00&\bf1.00 $\pm$ 0.00\\% 0.68 $\pm$ 0.46 & 0.72 $\pm$ 0.45 &0.77 $\pm$ 0.42 &   0.01 $\pm$ 0.10& 0.94 $\pm$ 0.24 & 0.98 $\pm$ 0.14 & 0.97 $\pm$ 0.17\\
			\hline
			HSIC&  0.03 $\pm$ 0.17 & 0.22 $\pm$ 0.42 &0.60 $\pm$ 0.49 & 0.86 $\pm$ 0.35 & 0.05 $\pm$ 0.22 &  0.26 $\pm$ 0.44  &  0.74 $\pm$ 0.44&\bf1.0 $\pm$ 0.00& 0.04 $\pm$ 0.20 & 0.16 $\pm$ 0.37 &0.98 $\pm$ 0.14& \bf1.00 $\pm$ 0.00 \\
			dCor$_{L2}$& 0.06 $\pm$ 0.24 &0.20 $\pm$ 0.40 & 0.59 $\pm$ 0.50 & 0.83 $\pm$ 0.38 & 0.03 $\pm$ 0.17 & 0.26 $\pm$ 0.44&  0.75 $\pm$ 0.44 &\bf1.0 $\pm$ 0.00&0.03 $\pm$ 0.17 & 0.17 $\pm$ 0.38 &0.96 $\pm$ 0.20& \bf1.00 $\pm$ 0.00 \\
      dCor$_{L1}$ & 0.02 $\pm$ 0.14 & 0.18 $\pm$ 0.39& 0.49 $\pm$ 0.50 & 0.72 $\pm$ 0.45 & 0.05 $\pm$ 0.22 & 0.18 $\pm$ 0.39&  0.55 $\pm$ 0.50 &0.80 $\pm$ 0.40& 0.08 $\pm$ 0.27 & 0.12 $\pm$ 0.33 &0.79 $\pm$ 0.41& 0.97 $\pm$ 0.17\\
			\hline
      %      \textbf{Model (M2)}  &\multicolumn{4}{c|}{\textbf{$N=500$, $d = 4$}   } & \multicolumn{4}{c|}{\textbf{$N=500$, $d=24$}  }& \multicolumn{4}{c|}{\textbf{$N=500$, $d=50$}  }  \\
       %     \hline
        %    \textbf{Rejection rate of the null} & \text{$\cH_0$} & \multicolumn{3}{c|}{\textbf{$\cH_1$}}& \text{$\cH_0$} & \multicolumn{3}{c|}{\textbf{$\cH_1$}}&\text{$\cH_0$} & \multicolumn{3}{c|}{\textbf{$\cH_1$}} \\
         % \textbf{Method }&  $\rho=0.0$ & $\rho=0.1$ & $\rho=0.3$ & $\rho=0.5$ & $\rho=0.0$ & $\rho=0.3$ & $\rho=0.5$ & $\rho=0.7$  & $\rho=0.0$ & $\rho=1$ & $\rho=2$ & $\rho=3$ \\
			%\hline 
			%\hline 
			%\bf $RF_{MWW}$  & 0.05 $\pm$ 0.22 &  1.00 $\pm$ 0.00  & 1.00 $\pm$ 0.00 &  0.99 $\pm$ 0.10 & 0.05 $\pm$ 0.22 & 0.05 $\pm$ 0.22 & 0.08 $\pm$ 0.27 & 0.08 $\pm$ 0.27  \\
       %     $RF_{95}$ & 0.01 $\pm$ 0.10 & 0.02 $\pm$ 0.14 & 0.00 $\pm$ 0.00 & 0.01 $\pm$ 0.10 &0.01 $\pm$ 0.10& 0.00 $\pm$ 0.00& 0.00 $\pm$ 0.00& 0.00 $\pm$ 0.00 \\
        %    $RF_{90}$&0.00 $\pm$ 0.00 & 0.80 $\pm$ 0.40 &0.77 $\pm$ 0.42&0.77 $\pm$ 0.42 & 0.00 $\pm$ 0.00& 0.00 $\pm$ 0.00 & 0.00 $\pm$ 0.00& 0.00 $\pm$ 0.00\\
         %   $RF_{85}$& 0.01 $\pm$ 0.10&  0.99 $\pm$ 0.10 &  0.99 $\pm$ 0.10 &  0.97 $\pm$ 0.17 & 0.01 $\pm$ 0.10 & 0.00 $\pm$ 0.00 & 0.00 $\pm$ 0.00& 0.00 $\pm$ 0.00\\
			%\hline
			%HSIC&  \\giit pul
			%dCor& \\
			
		%	\hline
		\end{tabular}
	}		
	\caption{Empirical rejection rates for testing $\cH_0$ of independence against $\cH_1$, for models (GL, M1) $\pm$ $95\%$ standard deviation at significance level $\alpha = 0.05$. Parameters: $\rho\in[0,0.6]$ (GL), $\rho\in\{0,1,2,3\}$ (M1), $d\in\{4,10, 26\}$, $K_p = 50$, $B=100$, $B_p=200$. Results in bold specify the best performance among all methods.}
	\label{tab:all500}
\end{table*}
\end{landscape}

\begin{landscape}
\begin{table*}[ht!]
	%\small
	\renewcommand{\arraystretch}{1.1}
	\centering
	\resizebox{2\textwidth}{!}{	
		\begin{tabular}{|l||c|ccc|c|ccc|c|ccc|}
			\hline
            \textbf{Model (M1s, $N=500$)}  &\multicolumn{4}{c|}{\textbf{$d = 50$}   } & \multicolumn{4}{c|}{\textbf{$d=100$}} & \multicolumn{4}{c|}{\textbf{$d=250$}}  \\
              \hline
              \textbf{Rejection rate of the null} & \text{$\cH_0$} & \multicolumn{3}{c|}{\textbf{$\cH_1$}}& \text{$\cH_0$} & \multicolumn{3}{c|}{\textbf{$\cH_1$}}&\text{$\cH_0$} & \multicolumn{3}{c|}{\textbf{$\cH_1$}} \\
          \textbf{Method}	& $\rho=0.0$  & $\rho=0.30$ & $\rho=0.40$ & $\rho=0.50$ & $\rho=0.0$  & $\rho=0.30$ & $\rho=0.40$ & $\rho=0.50$& $\rho=0.0$  & $\rho=0.20$ & $\rho=0.30$ & $\rho=0.40$ \\
			\hline 

			\hline 
			\bf rForest$_{MWW}$  & 0.03 $\pm$ 0.17 & \bf 0.94 $\pm$ 0.24& \bf0.96 $\pm$ 0.20& \bf0.97 $\pm$ 0.17& 0.03 $\pm$ 0.17   & \bf 0.82 $\pm$ 0.39&  \bf0.96 $\pm$ 0.20 &  \bf 1.00 $\pm$ 0.00 & 0.07 $\pm$ 0.26   & 0.11 $\pm$ 0.31 &  \bf0.90 $\pm$ 0.30 & 0.98 $\pm$ 0.14\\ %1.00 $\pm$ 0.00&& 1.00 $\pm$ 0.00& 0.05 $\pm$ 0.22 & 1.00 $\pm$ 0.00& run& run& run& run& run& run\\
      rForest$_{95}$ & 0.01 $\pm$ 0.10 &0.02 $\pm$ 0.14&0.03 $\pm$ 0.17 &0.74 $\pm$ 0.44 &0.01 $\pm$ 0.10 &0.00 $\pm$ 0.00& 0.05 $\pm$ 0.22 &  0.73 $\pm$ 0.45 &0.01 $\pm$ 0.10 &0.00 $\pm$ 0.00&0.00 $\pm$ 0.00&0.00 $\pm$ 0.00\\%0.84 $\pm$ 0.37&&0.97 $\pm$ 0.17 & 0.1 $\pm$ 0.10 &1.00 $\pm$ 0.00 \\
      rForest$_{90}$ &  0.01 $\pm$ 0.10 &0.32 $\pm$ 0.47 &0.60 $\pm$ 0.49 & 0.95 $\pm$ 0.22 &0.04 $\pm$ 0.19 &0.00 $\pm$ 0.00  & 0.60 $\pm$ 0.49 &  0.98 $\pm$ 0.14 & 0.02 $\pm$ 0.14 &0.00 $\pm$ 0.00&0.00 $\pm$ 0.00&0.06 $\pm$ 0.24\\%0.99$\pm$ 0.10 & & 1.00 $\pm$ 0.00& 0.03 $\pm$ 0.17& 1.00 $\pm$ 0.00\\
      rForest$_{85}$ & 0.01 $\pm$ 0.10   & 0.71 $\pm$ 0.46 &0.89 $\pm$ 0.31& \bf 0.97 $\pm$ 0.17&0.03 $\pm$ 0.17   & 0.00 $\pm$ 0.00 & 0.82 $\pm$ 0.39&  0.99 $\pm$ 0.10& 0.03 $\pm$ 0.17 &0.00 $\pm$ 0.00&0.06 $\pm$ 0.24&0.24 $\pm$ 0.42\\%1.00 $\pm$ 0.00& & 1.00 $\pm$ 0.00& 0.02 $\pm$ 0.14& 1.00 $\pm$ 0.00\\

			\hline
      HSIC & 0.05 $\pm$ 0.22 &  0.27 $\pm$ 0.45 &0.41 $\pm$ 0.49&0.72 $\pm$ 0.45& 0.03 $\pm$ 0.17 & 0.46 $\pm$ 0.50 &0.80 $\pm$ 0.40 &0.92 $\pm$ 0.27& 0.04 $\pm$ 0.19 &0.23 $\pm$ 0.42&0.74 $\pm$ 0.44& \bf1.00 $\pm$ 0.00\\

      dCor$_{L2}$ & 0.06 $\pm$ 0.24 &    0.28 $\pm$ 0.45 & 0.39 $\pm$ 0.49&0.73 $\pm$ 0.45& 0.05 $\pm$ 0.22 &  0.41 $\pm$ 0.49  &0.80 $\pm$ 0.40 &0.93 $\pm$ 0.26&0.04 $\pm$ 0.19 & \bf  0.24 $\pm$ 0.43&0.74 $\pm$ 0.44 & \bf1.00 $\pm$ 0.00\\
      dCor$_{L1}$ & 0.04 $\pm$ 0.19 &  0.21 $\pm$ 0.41 & 0.27 $\pm$ 0.45&0.59 $\pm$ 0.49& 0.02 $\pm$ 0.14 &  0.35 $\pm$ 0.48 & 0.58 $\pm$ 0.50&0.78 $\pm$ 0.42& 0.03 $\pm$ 0.17 & 0.22 $\pm$ 0.42 &0.58 $\pm$ 0.50&0.93 $\pm$ 0.27\\
      \hline
      \hline
            \textbf{Model (M1d, $N=500$)}  &\multicolumn{4}{c|}{\textbf{$d = 50$}   } & \multicolumn{4}{c|}{\textbf{$d=100$}} & \multicolumn{4}{c|}{\textbf{$d=250$}}  \\
              \hline
              \textbf{Rejection rate of the null} & \text{$\cH_0$} & \multicolumn{3}{c|}{\textbf{$\cH_1$}}& \text{$\cH_0$} & \multicolumn{3}{c|}{\textbf{$\cH_1$}}&\text{$\cH_0$} & \multicolumn{3}{c|}{\textbf{$\cH_1$}} \\
          \textbf{Method}	& $\rho=0.0$  & $\rho=0.20$ & $\rho=0.30$ & $\rho=0.40$ & $\rho=0.0$ & $\rho=0.15$ & $\rho=0.20$ & $\rho=0.30$ & $\rho=0.0$ &$\rho=0.20$ & $\rho=0.30$ & $\rho=0.40$ \\
			\hline 

			\hline 
			\bf rForest$_{MWW}$ & 0.03 $\pm$ 0.17 & \bf 0.37 $\pm$ 0.49&\bf  0.99$\pm$ 0.10& 0.96 $\pm$ 0.20 &0.03 $\pm$ 0.17 & 0.00 $\pm$ 0.00& 0.02 $\pm$ 0.14&\bf  0.99 $\pm$ 0.10 &0.07 $\pm$ 0.26   &0.00 $\pm$ 0.00& 0.97 $\pm$ 0.17 &\bf  1.00 $\pm$ 0.00 \\ 
      rForest$_{95}$ & 0.01 $\pm$ 0.10& 0.00 $\pm$ 0.00 &  0.10 $\pm$  0.31& 0.03 $\pm$ 0.17  &0.01 $\pm$ 0.10&0.00 $\pm$ 0.00 &0.00 $\pm$ 0.00& 0.02 $\pm$ 0.14&0.01 $\pm$ 0.10&0.00 $\pm$ 0.00& 0.00 $\pm$ 0.00  & 0.85 $\pm$ 0.36 \\
      rForest$_{90}$ &0.01 $\pm$ 0.10 &0.00 $\pm$ 0.00 & 0.65 $\pm$ 0.48 &  0.60 $\pm$ 0.49 &0.04 $\pm$ 0.19 &0.00 $\pm$ 0.00&0.00 $\pm$ 0.00&  0.32 $\pm$ 0.47&0.02 $\pm$ 0.14&0.00 $\pm$ 0.00&  0.00 $\pm$ 0.00 &0.98 $\pm$ 0.14 \\
      rForest$_{85}$ & 0.02 $\pm$ 0.14 & 0.00 $\pm$ 0.00 & 0.92 $\pm$ 0.39&  0.89$\pm$ 0.31 & 0.03 $\pm$ 0.17 &0.00 $\pm$ 0.00&0.00 $\pm$ 0.00&0.71 $\pm$ 0.46&0.03 $\pm$ 0.17 &0.00 $\pm$ 0.00&  0.06 $\pm$ 0.24 &  0.99$ \pm$ 0.10 \\

			\hline
      HSIC & 0.05 $\pm$ 0.22 & 0.32 $\pm$ 0.47 & 0.93 $\pm$ 0.26 &\bf 0.99 $\pm$ 0.10&0.03 $\pm$ 0.17 &0.27$\pm$ 0.45&0.43 $\pm$ 0.50&0.93 $\pm$ 0.26&0.04 $\pm$ 0.19 &0.74$\pm$ 0.44& \bf 1.00 $\pm$ 0.00& \bf 1.00 $\pm$ 0.00\\
      dCor$_{L2}$ &0.06 $\pm$ 0.24 &0.33 $\pm$ 0.47 &  0.95 $\pm$ 0.22 & \bf 0.99 $\pm$ 0.10& 0.05 $\pm$ 0.2 &\bf 0.29$\pm$ 0.46& \bf 0.47 $\pm$ 0.50&0.95 $\pm$ 0.22&0.04 $\pm$ 0.19 & 0.74$\pm$ 0.44&\bf1.00 $\pm$ 0.00& \bf 1.00 $\pm$ 0.00\\
      dCor$_{L1}$ &0.04 $\pm$ 0.19 & 0.19 $\pm$ 0.39 &  0.80 $\pm$ 0.40 &0.88 $\pm$ 0.33 & 0.02 $\pm$ 0.14 &0.15 $\pm$ 0.36& 0.32 $\pm$ 0.47& 0.80 $\pm$ 0.40&0.03 $\pm$ 0.17 & 0.68 $\pm$ 0.47&\bf1.00 $\pm$ 0.00& \bf 1.00 $\pm$ 0.00\\
      \hline
		\end{tabular}
	 }		
	\caption{Empirical rejection rates for  model (M1s, M1d) $\pm$ $95\%$ standard deviation  at significance level $\alpha = 0.05$. Parameters: $\rho\in\{0.0,0.1,0.2,0.3\}$, $d\in\{50,100,250\}$, $K_p\in\{10, 50\}$, $K_0=200$,  $B=100$. Results in bold specify the best performance among all methods.}
	\label{tab:M2_HDs}
    \label{tab:M2_HDd}
\end{table*}
\end{landscape}

\end{document}